\documentclass[11pt]{amsart}
\usepackage{amscd}
\usepackage{amsmath,empheq}
\usepackage{amsfonts}
\usepackage{amssymb}
\usepackage{mathrsfs}

\newtheorem{theorem}{Theorem}

\newtheorem{lemma}[theorem]{Lemma}
\newtheorem{prop}[theorem]{Proposition}
\newtheorem{remark}{Remark}

\newtheorem{claim}{Claim}

\newtheorem{definition}[theorem]{Definition}

\newenvironment{proof-sketch}{\noindent{\bf Sketch of Proof}\hspace*{1em}}{\qed\bigskip}

\newcommand{\RR}{\mathbb R}

\newcommand{\ZZ}{\mathbb Z}

\renewcommand{\leq}{\leqslant}

\renewcommand{\geq}{\geqslant}

\begin{document}

\title[On a class of parametric $(p,2)$-equations]{On a class of parametric $(p,2)$-equations}

\author[N.S. Papageorgiou]{N.S. Papageorgiou}
\address[N.S. Papageorgiou]{National Technical University, Department of Mathematics,
				Zografou Campus, 15780 Athens, Greece 
				\& Institute of Mathematics, Physics and Mechanics, 1000 Ljubljana, Slovenia}
\email{\tt npapg@math.ntua.gr}
\author[V.D. R\u{a}dulescu]{V.D. R\u{a}dulescu}
\address[V.D. R\u{a}dulescu]{Institute of Mathematics, Physics and Mechanics, 1000 Ljubljana, Slovenia \&  Faculty of Applied Mathematics, AGH University of Science and Technology, 30-059 Krak\'ow, Poland}
\email{\tt vicentiu.radulescu@imfm.si}
\author[D.D. Repov\v{s}]{D.D. Repov\v{s}}
\address[D.D. Repov\v{s}]{Faculty of Education and Faculty of Mathematics and Physics,
University of Ljubljana, \& Institute of Mathematics, Physics and Mechanics, 1000 Ljubljana, Slovenia}\email{dusan.repovs@guest.arnes.si}


\begin{abstract}
We consider parametric equations driven by the sum of a $p$-Laplacian and a Laplace operator (the so-called $(p,2)$-equations). We study the existence and multiplicity of solutions when the parameter $\lambda>0$ is near the principal eigenvalue $\hat{\lambda}_1(p)>0$ of $(-\Delta_p,W^{1,p}_{0}(\Omega))$. We prove multiplicity results with precise sign information when the near resonance occurs from above and from below of $\hat{\lambda}_1(p)>0$.
\end{abstract}

\maketitle


\section{Introduction}\label{sec1}

Let $\Omega\subseteq\RR^N$ be a bounded domain with a $C^2$-boundary $\partial\Omega$. In this paper, we study the following parametric nonlinear nonhomogeneous Dirichlet problem
\begin{equation}
-\Delta_pu(z)-\Delta u(z)=\lambda|u(z)|^{p-2}u(z)+f(z,u(z))\ \mbox{in}\ \Omega,\ u|_{\partial\Omega}=0,\ 2<p<\infty.  \tag{$P_{\lambda}$} \label{eqP}
\end{equation}

Here $\Delta_p$ denotes the $p$-Laplacian differential operator defined by
$$\Delta_pu=\mbox{div}\,(|Du|^{p-2}Du)\ \mbox{for all}\ u\in W^{1,p}_0(\Omega).$$

Also, $\lambda>0$ is a parameter and $f:\Omega\times\RR\rightarrow\RR$ is a Carath\'eodory perturbation (that is, for all $x\in\RR$, $z\longmapsto f(z,x)$ is measurable and for a.a. $z\in\Omega$, $x\longmapsto f(z,x)$ is continuous).

Our aim in this paper is to study the existence and multiplicity of nontrivial solutions when the parameter $\lambda>0$ is near the principal eigenvalue $\hat{\lambda}_1(p)>0$ of $(-\Delta_p,W^{1,p}_{0}(\Omega))$ either from the left or from the right. Such equations, which are near resonance, were first investigated by Mawhin and Schmitt \cite{18}, \cite{19} (for semilinear Dirichlet and periodic problems, respectively). Subsequently, their work was extended by Badiale and Lupo \cite{4}, Chiappinelli, Mawhin and Nugari \cite{8} and Ramos and Sanchez \cite{30}. All these papers consider semilinear elliptic equations driven by the Laplacian. Extensions to equations driven by the $p$-Laplacian were obtained by Ma, Ramos and Sanchez \cite{17} and Papageorgiou and Papalini \cite{22}.

In this work we extend the analysis to $(p,2)$-equations (that is, equations driven by the sum of a $p$-Laplacian $(p>2)$ and a Laplacian). We stress that the differential operator in \eqref{eqP} is nonhomogeneous and this is a source of difficulties in the analysis of the problem \eqref{eqP}. We note that $(p,2)$-equations arise in many physical applications (see Cherfils and Ilyasov \cite{7}) and recently such equations were studied by Barile and Figueiredo \cite{barile}, Carvalho, Goncalves and da Silva \cite{carval}, Chaves, Ercole and Miyagaki \cite{chaves}, Mugnai and Papageorgiou \cite{20}, Papageorgiou and R\u{a}dulescu \cite{23}, \cite{24}, \cite{25} and Papageorgiou and Winkert \cite{27}, \cite{28}.

Our approach is variational, based on the critical point theory, together with suitable truncation and comparison techniques, and Morse theory (critical groups). In the next section, for the convenience of the reader, we recall the main mathematical tools which we will use in the paper.

\section{Mathematical Background}\label{sec2}

The topological notion of linking sets is central in the critical point theory.
\begin{definition}\label{def1}
	Let $Y$ be a Hausdorff topological space and $E_0,\, E,D$ be closed subspaces of $Y$ such that $E_0\subseteq E$. We say that the pair $\{E_0,E\}$ is linking with $D$ in $Y$, if
	\begin{itemize}
		\item[(a)] $E_0\cap D=\emptyset $; and
		\item[(b)] for every $\gamma\in C(E,Y)$ such that $\gamma|_{E_0}=id|_{E_0}$, we have $\gamma(E)\cap D\neq\emptyset $.
	\end{itemize}
\end{definition}
	
	Now, let $X$ be a Banach space and $X^*$ its topological dual. By $\left\langle \cdot,\cdot\right\rangle$ we denote the duality brackets for the pair $(X,X^*)$. Given $\varphi\in C^1(X)$, we say that $\varphi$ satisfies the Cerami condition (the $C$-condition for short), if the following is true:\\
		``If $\{u_n\}_{n\geq 1}\subseteq X$ is a sequence such that $\{\varphi(u_n)\}_{n\geq 1}\subseteq\RR$ is bounded and
		$$(1+||u_n||)\varphi'(u_n)\rightarrow 0\ \mbox{in}\ W^{-1,p'}(\Omega)=W^{1,p}_{0}(\Omega)^*\left(\frac{1}{p}+\frac{1}{p'}=1\right)\ \mbox{as}\ n\rightarrow\infty,$$
		then it admits a strongly convergent subsequence".

This is a compactness-type condition on the functional $\varphi$, which compensates for the fact that the ambient space $X$ need not be locally compact (since $X$ is in general, infinite dimensional). The $C$-condition is important in developing a minimax theory for the critical values of $\varphi$. A basic result in that theory is the following theorem which involues the notion of linking sets (see, for example, Gasinski and Papageorgiou \cite[p. 644]{13}).
\begin{theorem}\label{th2}
	If $X$ is a Banach space, $E_0,E$ and $D$ are nonempty closed subsets of $X$ such that the pair $\{E_0,E\}$ is linking with $D$ in $X$ (see Definition \ref{def1}), $\varphi\in C^1(X)$ and satisfies the $C$-condition, $\sup_{E_0} \varphi<\inf_D \varphi$ and
	$$c=\inf_{\gamma\in\Gamma}\sup_{u\in E}\ \varphi(\gamma(u))\ \mbox{with}\ \Gamma=\{\gamma\in C(E,X):\gamma|_{E_0}=id|_{E_0}\},$$
	then $c\geq \inf_D \varphi$ and $c$ is a critical value of $\varphi$.
\end{theorem}

With suitable choices of the linking sets, we obtain the well-known mountain pass theorem, saddle point theorem and the generalized mountain pass theorem (see \cite{13}). For future use, we state the mountain pass theorem.
\begin{theorem}\label{th3}
	If $X$ is a Banach space, $\varphi\in C^1(X)$ and satisfies the $C$-condition, $u_0,u_1\in X$
	$$\max\{\varphi(u_0),\varphi(u_1)\}<\inf\left[\varphi(u):||u-u_0||=\rho\right]=m_{\rho},\ ||u_1-u_0||>\rho>0$$
	and $c=\displaystyle\inf_{\gamma\in\Gamma}\max_{0\leq t\leq 1} \varphi(\gamma(t))$ with $\Gamma=\{\gamma\in C([0,1],X):\gamma(0)=u_0,\ \gamma(1)=u_1\}$, then $c\geq m_{\rho}$ and $c$ is a critical value of $\varphi$.
\end{theorem}
\begin{remark}
	It is easy to see that Theorem \ref{th3} can be deduced from Theorem \ref{th2}, if we consider\\
	$E_0=\{u_0,u_1\},\ E=\{u\in X:u=tu_1+(1-t)u_0,t\in[0,1]\},\ D=\partial B_{\rho}(u_0)=\{u\in X:||u-u_0||=\rho\}$.
\end{remark}

In this analysis of problem \eqref{eqP}, we will use the Sobolev space $W^{1,p}_{0}(\Omega)$ and the Banach space $C^{1}_{0}(\overline{\Omega})=\left\{u\in C^1(\overline{\Omega}):u|_{\partial\Omega}=0\right\}$. The latter is an ordered Banach space with positive cone $C_+=\{u\in C^{1}_{0}(\overline{\Omega});u(z)\geq 0\ \mbox{for all}\ z\in\overline{\Omega}\}$. This cone has nonempty interior given by
$$\mbox{int}\, C_+=\left\{u\in C_+:u(z)>0\ \mbox{for all}\ z\in\Omega,\ \left.\frac{\partial u}{\partial n}\right|_{\partial\Omega}<0\right\}.$$

Here $n(\cdot)$ denotes the outward unit normal on $\partial\Omega$.

In what follows, by $||\cdot||$ we denote the norm of the Sobolev space $W^{1,p}_{0}(\Omega)$. By virtue of the Poincar\'e inequality, we have
$$||u||=||Du||_p\ \mbox{for all}\ u\in W^{1,p}_{0}(\Omega).$$

Next, we present some basic facts about the spectrum of $(-\Delta_q,W^{1,q}_{0}(\Omega))$ with $1<q<\infty$. So, we consider the following nonlinear eigenvalue problem
$$-\Delta_qu(z)=\hat{\lambda}|u(z)|^{q-2}u(z)\ \mbox{in}\ \Omega,\ u|_{\partial\Omega}=0.$$

We say that $\hat{\lambda}\in\RR$ is an eigenvalue of $(-\Delta_q,W^{1,q}_{0}(\Omega))$, if the above equation admits a nontrivial solution $\hat{u}\in W^{1,q}_{0}(\Omega)$. We say that $\hat{u}$ is an eigenfunction corresponding to the eigenvalue $\hat{\lambda}$. We know that there exists a smallest eigenvalue $\hat{\lambda}_1(q)$ with the following properties:
\begin{itemize}
	\item[(i)] $\hat{\lambda}_1(q)>0$;
	\item[(ii)] $\hat{\lambda}_1(q)$ is isolated, that is, there exists $\epsilon>0$ such that $(\hat{\lambda}_1(q),\hat{\lambda}_1(q)+\epsilon)$ contains no eigenvalue of $(-\Delta_q,W^{1,q}_{0}(\Omega))$; and
	\item[(iii)] $\hat{\lambda}_1(q)$ is simple, that is, if $\hat{u},\hat{v}$ are eigenfunctions corresponding to $\hat{\lambda}_1(q)$, then $\hat{u}=\xi\hat{v}$ for some $\xi\in\RR\backslash\{0\}$.
\end{itemize}

Moreover, $\hat{\lambda}_1(q)$ admits the following variational characterization
\begin{equation}\label{eq1}
	\hat{\lambda}_1(q)=\inf\left[\frac{||Du||^{q}_{q}}{||u||^{q}_{q}}:u\in W^{1,q}_{0}(\Omega),\ u\neq 0\right].
\end{equation}

In (\ref{eq1}) the infimum is realized on the corresponding one-dimensional eigenspace. By (\ref{eq1}) it is clear that the elements of this eigenspace do not change the sign. By $\hat{u}_1(q)$ we denote the positive, $L^p$-normalized (that is, $||\hat{u}_1(q)||_q=1$) eigenfunction corresponding to $\hat{\lambda}_1(q)>0$. From the nonlinear regularity theory and the nonlinear maximum principle (see, for example, Gasinski and Papageorgiou \cite[pp. 737-738]{13}), it follows that $\hat{u}_1(q)\in \mbox{int}\, C_+$.

Let $\sigma(q)$ denote the set of eigenvalues of $(-\Delta_q,W^{1,q}_{0}(\Omega))$. It is easy to check that this set is closed. Since $\hat{\lambda}_1(q)>0$ is isolated, the second eigenvalue $\hat{\lambda}^{*}_{2}(q)$ is well-defined by
$$\hat{\lambda}^{*}_{2}(q)=\inf[\hat{\lambda}\in\sigma(q):\hat{\lambda}>\hat{\lambda}_1(q)].$$

If $N=1$ (ordinary differential equations), then $\sigma(q)=\{\hat{\lambda}_k(q)\}_{k\geq 1}$ with each $\hat{\lambda}_k(q)$ being a simple eigenvalue and $\hat{\lambda}_k(q)\uparrow +\infty$ as $k\rightarrow\infty$ and the corresponding eigenfunctions $\{\hat{u}_k(q)\}_{k\geq 1}$ have exactly $k-1$ zeros. If $N\geq 2$ (partial differential equations), then using the Ljusternik-Schnirelmann minimax scheme, we can produce a strictly increasing sequence $\{\hat{\lambda}_k(q)\}_{k\geq 1}\subseteq\sigma(q)$ such that $\hat{\lambda}_k(q)\rightarrow+\infty$ as $k\rightarrow\infty$. However, we do not know if this is the complete list of all eigenvalues. We know that $\hat{\lambda}^{*}_{2}(q)=\hat{\lambda}_2(q)$, that is, the second eigenvalue and the second Ljusternik-Schnirelmann eigenvalue coincide. The Ljusternik-Schnirelmannn theory gives a minimax  characterization of $\hat{\lambda}_2(q)$. For our purposes, this characterization is not convenient. Instead, we will us an altern
 ative one due to Cuesta, de Figueiredo and Gossez \cite{10}.
\begin{prop}\label{prop4}
	If $\partial B^{L^q}_{1}=\{u\in L^q(\Omega):||u||_q=1\},\ M=W^{1,q}_{0}(\Omega)\cap \partial B^{L^q}$, and
	$$\Gamma_0=\{\gamma_0\in C([-1,1],M):\gamma_0(-1)=-\hat{u}_1(q),\ \gamma_0(1)=\hat{u}_1(q)\}$$
	then $\hat{\lambda}_2(q)=\inf_{\gamma_0\in\Gamma_0}\max_{-1\leq t\leq 1}||D\gamma_0(t)||^{q}_{q}$.
\end{prop}

We mention that $\hat{\lambda}_1(q)>0$ is the only eigenvalue with eigenfunctions of constant sign. Every other eigenvalue has nodal (that is, sign-changing) eigenfunctions.

When $q=2$ (linear eigenvalue problem), then $\sigma(2)=\{\hat{\lambda}_k(2)\}_{k\geq 1}$. In this case, the eigenspaces are linear spaces. By $E(\hat{\lambda}_k(2))$, we denote the eigenspace corresponding to the eigenvalue $\hat{\lambda}_k(2)$. The regularity theory implies that $E(\hat{\lambda}_k(2))\subseteq C^{1}_{0}(\overline{\Omega})$. Moreover, $E(\hat{\lambda}_k(2))$ has the so-called unique continuation property, that is, if $u\in E(\hat{\lambda}_k(2))$ and vanishes on a set of positive Lebesgue measure, then $u\equiv 0$. In this case all eigenvalues admit variational characterization, namely
\begin{equation}\label{eq2}
	\hat{\lambda}_1(2)=\inf\left[\frac{||Du||^{2}_{2}}{||u||^{2}_{2}}:u\in H^{1}_{0}(\Omega),u\neq 0\right]
\end{equation}
and for $k\geq 2$, we have
\begin{eqnarray}\label{eq3}
	\hat{\lambda}_k(2)&=&\sup\left[\frac{||Du||^{2}_{2}}{||u||^{2}_{2}}:u\in \overset{k}{\underset{\mathrm{i=1}}\oplus}E(\hat{\lambda}_i(2)),u\neq 0\right]\nonumber\\
	&=&\inf\left[\frac{||Du||^{2}_{2}}{||u||^{2}_{2}}:u\in\overline{{\underset{\mathrm{i\geq k}}\oplus}E(\hat{\lambda}_i(2))},u\neq 0\right].
\end{eqnarray}

In (\ref{eq2}) the infimum is realized on $E(\hat{\lambda}_1(2))$, while in (\ref{eq3}) both the supremum and the infimum are realized on $E(\hat{\lambda}_k(2))$.

From the variational characterizations in (\ref{eq2}) and (\ref{eq3}) and the unique continuation property, we have the following result (see Papageorgiou and Kyritsi \cite{21}).
\begin{prop}\label{prop5}
	\begin{itemize}
		\item[(a)] If $k\geq 1,\ \vartheta\in L^{\infty}(\Omega)$, $\vartheta(z)\leq\hat{\lambda}_k(2)$ for a.a. $z\in\Omega$ and $\vartheta\not\equiv\hat{\lambda}_k(2)$, then there exists $\hat{\xi}_0>0$ such that
		$$||Du||^{2}_{2}-\int_{\Omega}\vartheta(z)u(z)^2dz\geq \hat{\xi}_0||u||^2\ \mbox{for all}\ u\in\overline{{\underset{\mathrm{i\geq k}}\oplus}E({\hat{\lambda}_k(2))}}.$$
		\item[(b)] If $k\geq 1$, $\vartheta\in L^{\infty}(\Omega)$, $\vartheta(z)\geq\hat{\lambda}_k(2)$ for a.a. $z\in\Omega$ and $\vartheta\not\equiv\hat{\lambda}_k(2)$, then there exists $\hat{\xi}_1>0$ such that
		$$||Du|||^{2}_{2}-\int_{\Omega}\vartheta(z)u(z)^2dz\leq-\hat{\xi}_1||u||^2\ \mbox{for all}\ u\in{\overset{k}{\underset{\mathrm{i=1}}\oplus}E(\hat{\lambda}_i(2))}.$$
	\end{itemize}
\end{prop}

For $1<q<\infty$, let $A_q:W^{1,q}_{0}(\Omega)\rightarrow W^{-1,q'}(\Omega)$ be the nonlinear map defined by
$$\left\langle A_q(u),h\right\rangle=\int_{\Omega}|Du|^{q-2}(Du,Dh)_{\RR^N}dz\ \mbox{for all}\ u,h\in W^{1,q}_{0}(\Omega).$$

If $q=2$, then $A_2=A\in\mathcal{L}(H^{1}_{0}(\Omega),H^{-1}(\Omega))$.

By Papageorgiou and Kyritsi \cite[p. 314]{21}, we have the following result summarizing the basic properties of the map $A_q$.
\begin{prop}\label{prop6}
	The map $A_q:W^{1,q}_{0}(\Omega)\rightarrow W^{-1.q'}(\Omega)$ is bounded (that is, it maps bounded sets to bounded sets), demicontinuous, strictly monotone (hence maximal monotone, too) and of type $(S)_+$, that is, if $u_n\stackrel{w}{\rightarrow} u$ in $W^{1,q}_{0}(\Omega)$ and $\limsup\limits_{n\rightarrow\infty}\left\langle A_q(u_n),u_n-u\right\rangle\leq 0$, then $u_n\rightarrow u$ in $W^{1,q}_{0}(\Omega)$ as $n\rightarrow\infty$.
\end{prop}

Let $f_0:\Omega\times\RR\rightarrow\RR$ be a Carath\'eodory function with subcritical growth in the $x\in\RR$ variable, that is,
$$|f_0(z,x)|\leq a_0(z)(1+|x|^{r-1})\ \mbox{for a.a.}\ z\in\Omega,\ \mbox{all}\ x\in\RR,$$
with $a_0\in L^{\infty}(\Omega)_+$ and $1<r<p^*=\left\{
\begin{array}{ll}
	\frac{Np}{N-p}&\mbox{if}\ p<N\\
	+\infty&\mbox{if}\ N\leq p.
\end{array}\right.$

We set $F_0(z,x)=\int^{x}_{0}f_0(z,s)ds$ and consider the $C^1$-functional $\varphi_0:W^{1,p}_{0}(\Omega)\rightarrow\RR$ defined by
$$\varphi_0(u)=\frac{1}{p}||Du||^{p}_{p}+\frac{1}{2}||Du||^{2}_{2}-\int_{\Omega}F_0(z,u(z))dz\ \mbox{for all}\ u\in W^{1,p}_{0}(\Omega).$$

The next result is a special case of a more general result of Aizicovici, Papageorgiou and Staicu \cite{2}.
\begin{prop}\label{prop7}
	Let $u_0\in W^{1,p}_{0}(\Omega)$ be a local $C^1(\overline{\Omega})$-minimizer of $\varphi_0$, that is, there exists $\rho_0>0$ such that
	$$\varphi_0(u_0)\leq\varphi_0(u_0+h)\ \mbox{for all}\ h\in C^{1}_{0}(\overline{\Omega}),\ ||h||_{C^{1}_{0}(\overline{\Omega})}<\rho_0.$$
	Then $u_0\in C^{1,\alpha}_{0}(\overline{\Omega})$ for some $\alpha\in(0,1)$ and it is also a local $W^{1,p}_{0}(\Omega)$-minimizer of $\varphi_0$, that is, there exists $\rho_1>0$ such that
	$$\varphi_0(u_0)\leq\varphi_0(u_0+h)\ \mbox{for all}\ h\in W^{1,p}_0(\Omega),\ ||h||\leq\rho_1.$$
\end{prop}

We also recall some basic definitions and facts from Morse theory. So, let $\varphi\in C^1(X)$ and $c\in\RR$. We introduce the following sets.
$$\varphi^c=\{u\in X:\varphi(u)\leq c\},\ K_{\varphi}=\{u\in X:\varphi'(u)=0\}\ \mbox{and}\ K^{c}_{\varphi}=\{u\in K_{\varphi}:\varphi(u)=c\}.$$

Let $(Y_1,Y_2)$ be a topological pair with $Y_2\subseteq Y_1\subseteq X$. For every integer $k\geq 0$, by $H_k(Y_1,Y_2)$ we denote the $k$-th relative singular homology group with integer coefficients. The critical groups of $\varphi$ at $u\in K^{c}_{\varphi}$ which is isolated among the critical points, are defined by
$$C_k(\varphi,u)=H_k(\varphi^c\cap U,\ \varphi^c\cap U\backslash\{u\})\ \mbox{for all}\ k\geq 0.$$

Here $U$ is a neighborhood of $u$ such that $K_{\varphi}\cap \varphi^c\cap U=\{u\}$. The excision property of the singular homology implies that this definition is independent of the particular choice of the neighborhood $U$.

Suppose that $\varphi\in C^1(X)$ satisfies the $C$-condition and $\inf\ \varphi(K_{\varphi})>-\infty$. Let $c<\inf \varphi(K_{\varphi})$. Then the critical groups of $\varphi$ at infinity are defined by
$$C_k(\varphi,\infty)=H_k(X,\varphi^c)\ \mbox{for all}\ k\geq 0.$$

The second deformation theorem (see, for example, Gasinski and Papageorgiou \cite[p. 628]{13}), implies that this definition is independent of the choice of the level $c<\inf \varphi(K_{\varphi})$.

We introduce
\begin{eqnarray*}
	&&M(t,u)=\sum\limits_{k\geq 0}\mbox{rank}\, C_k(\varphi,u)t^k\ \mbox{for all}\ t\in\RR,\ \mbox{all}\ u\in K_{\varphi}\ \mbox{and}\\
	&&P(t,\infty)=\sum\limits_{k\geq 0}\mbox{rank}\, C_k(\varphi,\infty)t^k\ \mbox{for all}\ t\in\RR.
\end{eqnarray*}

The Morse relation says that
\begin{equation}\label{eq4}
	\sum\limits_{u\in K_{\varphi}}M(t,u)=P(t,\infty)+(1+t)Q(t)
\end{equation}
where $Q(t)=\sum\limits_{k\geq 0}\beta_kt^k$ is a formal series in $t\in\RR$ with nonnegative coefficients.

Finally, let us fix our notation in this paper. By $|\cdot|_N$ we denote the Lebesgue measure on $\RR^N$. Given $x\in\RR$, we let $x^{\pm}=\max\{\pm x,0\}$. Then for $u\in W^{1,p}_{0}(\Omega)$ we define $u^{\pm}(\cdot)=u(\cdot)^{\pm}$. We know that
$$u^{\pm}\in W^{1,p}_{0}(\Omega),\ u=u^+-u^-,\ |u|=u^++u^-.$$

Given a measurable function $g(z,x)$ (for example, a Carath\'eodory function), we set
$$N_g(u)(\cdot)=g(\cdot,u(\cdot))\ \mbox{for all}\ u\in W^{1,p}_{0}(\Omega)$$
(the Nemytski map corresponding to $g$). Evidently, $z\mapsto N_g(u)(z)$ is measurable.

\section{Near Resonance from the left of $\hat{\lambda}_1(p)>0$}\label{sec3}

In this section we deal with problem \eqref{eqP} in which the parameter is close to $\hat{\lambda}_1(p)>0$ from the left (near resonance from the left). We introduce the following conditions on the perturbation $f(z,x)$:

$H_1:$ $f:\Omega\times\RR\rightarrow\RR$ is a Carath\'eodory function such that $f(z,0)=0$ for a.a. $z\in\Omega$ and
\begin{itemize}
	\item[(i)] for every $\rho>0$, there exists $a_{\rho}\in L^{\infty}(\Omega)_+$ such that
	$$|f(z,x)|\leq a_{\rho}(z)\ \mbox{for a.a.}\ z\in\Omega,\ \mbox{all}\ |x|\leq\rho;$$
	\item[(ii)] $\lim\limits_{x\rightarrow\pm\infty}\ \frac{f(z,x)}{|x|^{p-2}x}=0$ uniformly for a.a. $z\in\Omega$ and if $F(z,x)=\int^{x}_{0}f(z,s)ds$, then
	$$\lim\limits_{x\rightarrow\pm\infty}\ \frac{F(z,x)}{x^2}=+\infty\ \mbox{uniformly for a.a.}\ z\in\Omega;\ \mbox{and}$$
	\item[(iii)] there exist an integer $m\geq 2$ and a function $\eta\in L^{\infty}(\Omega)$ such that
	\begin{eqnarray*}
		&&\eta(z)\in[\hat{\lambda}_m(2),\hat{\lambda}_{m+1}(2)]\ \mbox{for a.a.}\ z\in\Omega,\ \eta\not\equiv\hat{\lambda}_m(2),\ \eta\not\equiv\hat{\lambda}_{m+1}(2)\\
		&&\lim\limits_{x\rightarrow 0}\frac{f(z,x)}{x}=\eta(z)\ \mbox{uniformly for a.a.}\ z\in\Omega.
	\end{eqnarray*}
\end{itemize}
\begin{remark}
	Evidently, $f(z,\cdot)$ is differentiable at $x=0$ and $f'_x(z,0)=\eta(z)$. Hypothesis $H_1$ imply that there exists $c_1>0$ such that $F(z,x)\geq -c_1x^2$ for a.a. $z\in\Omega$, all $x\in\RR$.
\end{remark}
	
	For $\lambda>0$, let $\varphi_{\lambda}:W^{1,p}_{0}(\Omega)\rightarrow\RR$ be the energy functional for problem \eqref{eqP}, defined by
	 $$\varphi_{\lambda}(u)=\frac{1}{p}||Du||^{p}_{p}+\frac{1}{2}||Du||^{2}_{2}-\frac{\lambda}{p}||u||^{p}_{p}-\int_{\Omega}F(z,u(z))dz\ \mbox{for all}\ u\in W^{1,p}_{0}(\Omega).$$
	
	Evidently, $\varphi_{\lambda}\in C^1(W^{1,p}_{0}(\Omega))$.

\begin{prop}\label{prop8}
	If hypotheses $H_1(i),(ii)$ hold and $\lambda\in(0,\hat{\lambda}_1(p))$, then the functional $\varphi_{\lambda}$ is coercive.
\end{prop}
\begin{proof}
	By virtue of hypotheses $H_1(i),(ii)$, given $\epsilon>0$, we can find $c_2=c_2(\epsilon)>0$ such that
	\begin{equation}\label{eq5}
		F(z,x)\leq\frac{\epsilon}{p}|x|^p+c_2\ \mbox{for a.a.}\ z\in\Omega,\ \mbox{all}\ x\in\RR.
	\end{equation}
	
	Then for all $u\in W^{1,p}_{0}(\Omega)$, we have
	\begin{eqnarray*}
		 \varphi_{\lambda}(u)&=&\frac{1}{p}||Du||^{p}_{p}+\frac{1}{2}||Du||^{2}_{2}-\frac{\lambda}{p}||u||^{p}_{p}-\int_{\Omega}F(z,u(z))dz\\
		&\geq&\frac{1}{p}\left[1-\frac{\lambda+\epsilon}{\hat{\lambda}_1(p)}\right]\ ||u||^p-c_2|\Omega|_N\ (\mbox{see (\ref{eq1}) and (\ref{eq4})}).
	\end{eqnarray*}
	
	Choosing $\epsilon\in(0,\hat{\lambda}_1(p)-\lambda)$ (recall that $\lambda<\hat{\lambda}_1(p)$), we can conclude from the last inequality that $\varphi_{\lambda}$ is coercive.
\end{proof}

Let $V=\{u\in W^{1,p}_{0}(\Omega):\int_{\Omega}u\ \hat{u}_1(p)^{p-1}dz=0\}$ (recall $\hat{u}_1(p)\in \mbox{int}\, C_+$). We have
$$W^{1,p}_{0}(\Omega)=\RR\hat{u}_1(p)\oplus V.$$

We introduce the following quantity
$$\hat{\lambda}_V(p)=\inf\left[\frac{||Du||^{p}_{p}}{||u||^{p}_{p}}:u\in V,\ u\neq 0\right].$$
\begin{lemma}\label{lem9}
	$\hat{\lambda}_1(p)<\hat{\lambda}_V(p)\leq \hat{\lambda}_2(p)$.
\end{lemma}
\begin{proof}
	Clearly, $\hat{\lambda}_1(p)\leq\hat{\lambda}_V(p)$ (see (\ref{eq1})). Suppose that $\hat{\lambda}_1(p)=\hat{\lambda}_V(p)$. Then we can find $\{u_n\}_{n\geq 1}\subseteq V$ such that
	$$||u_n||_p=1\ \mbox{and}\ ||Du||^{p}_{p}\rightarrow \hat{\lambda}_V(p)=\hat{\lambda}_1(p).$$
	
	By passing to a suitable subsequence if necessary, we may assume that
	$$u_n\stackrel{w}{\rightarrow} u\ \mbox{in}\ W^{1,p}_{0}(\Omega)\ \mbox{and}\ u_n\rightarrow u\ \mbox{in}\ L^p(\Omega)\ \mbox{as}\ n\rightarrow\infty.$$
	
	We have $u\in V$ and $||u||_p=1$. Also,
	\begin{eqnarray*}	 &&\hat{\lambda}_1(p)\leq||Du||^{p}_{p}\leq\liminf\limits_{n\rightarrow\infty}||Du_n||^{p}_{p}=
\hat{\lambda}_V(p)=\hat{\lambda}_1(p),\\
		&\Rightarrow&\hat{\lambda}_1(p)=||Du||^{p}_{p},\ \mbox{hence}\ u=\pm\hat{u}_1(p)\ (\mbox{recall}\ ||u||_p=1).
	\end{eqnarray*}
	
	But then $u\notin V$, a contradiction. So, we have proved that
	$$\hat{\lambda}_1(p)<\hat{\lambda}_V(p).$$
	
	Next, suppose that $\hat{\lambda}_2(p)<\hat{\lambda}_V(p)$. By virtue of Proposition \ref{prop4}, we can find $\hat{\gamma}_0=\Gamma_0$ such that
	\begin{equation}\label{eq6}
		||D\hat{\gamma}_0(t)||^{p}_{p}<\hat{\lambda}_V(p)\ \mbox{for all}\ t\in[0,1].
	\end{equation}
	
	We have $\hat{\gamma}_0(-1)=-\hat{u}_1(p)$, $\hat{\gamma}_0(1)=\hat{u}_1(p)$. Consider the function $[-1,1]\ni t\rightarrow\sigma(t)=\int_{\Omega}\hat{\gamma}_0(t)\hat{u}_1(p)^{p-1}dz$. Evidently, this function is continuous and $\sigma(-1)=-||\hat{u}_1(p)||^{p}_{p}<0<||\hat{u}_1(p)||^{p}_{p}=\sigma(1)$. So, by Bolzano's theorem, we can find $t_0\in(0,1)$ such that
	\begin{eqnarray*}
		&&\sigma(t_0)=\int_{\Omega}\hat{\gamma}_0(t_0)\hat{u}_1(p)^{p-1}dz=0\\
		&\Rightarrow&\hat{\gamma}_0(t_0)\in V,\ \mbox{which contradicts (\ref{eq6})}.
	\end{eqnarray*}
	
	Therefore we infer that $\hat{\lambda}_V(p)\leq \hat{\lambda}_2(p)$.
\end{proof}

\begin{prop}\label{prop10}
	If hypotheses $H_1(i),(ii)$ hold and $\lambda=\hat{\lambda}_1(p)$, then $\left.\varphi_{\lambda}\right|_V$ is bounded from below.
\end{prop}
\begin{proof}
	Let $v\in V$. We have
	\begin{eqnarray}\label{eq7}		 \varphi_{\lambda}(v)&=&\frac{1}{p}||Dv||^{p}_{p}+\frac{1}{2}||Dv||^{2}_{2}-\frac{\hat{\lambda}_1(p)}{p}||v||^{p}_{p}-\int_{\Omega}F(z,v)dz\nonumber\\
		&\geq&\frac{\hat{\lambda}_V(p)-\hat{\lambda}_1(p)-\epsilon}{2}\ ||v||^{p}_{p}-c_2|\Omega|_N\ (\mbox{see (\ref{eq4})}).
	\end{eqnarray}
	
	From Lemma \ref{lem9} we know that $\hat{\lambda}_1(p)<\hat{\lambda}_V(p)$. So, we choose $\epsilon\in(0,\hat{\lambda}_v(p)-\hat{\lambda}_1(p))$. Then from (\ref{eq7}) we infer that $\left.\varphi_{\lambda}\right|_V$ with $\lambda=\hat{\lambda}_1(p)$, is bounded from below.
\end{proof}

Let $m_1=\inf\limits_V\varphi_{\hat{\lambda}_1(p)}>-\infty$ (see Proposition \ref{prop10}). Note that, if $\lambda\in(0,\hat{\lambda}_1(p))$ then
\begin{eqnarray}\label{eq8}
	&&\varphi_{\hat{\lambda}_1(p)}\leq \varphi_{\lambda},\nonumber\\
	&\Rightarrow&m_1\leq\inf\limits_V\varphi_{\lambda}\ \mbox{for all}\ \lambda\in(0,\hat{\lambda}_1(p)).
\end{eqnarray}
\begin{prop}\label{prop11}
	If hypothesis $H_1$ holds, then we can find small $\epsilon>0$ such that every $\lambda\in(\hat{\lambda}_1(p)-\epsilon, \hat{\lambda}_1(p))$ we can find large $t_0>0$  such that
	$$\varphi_{\lambda}(\pm t_0\hat{u}_1(p))<m_1.$$
\end{prop}
\begin{proof}
	By virtue of hypothesis $H_1(ii)$, given $\xi>0$, we can find $M_1=M_1(\xi)>0$ such that
	\begin{equation}\label{eq9}
		F(z,x)\geq\xi x^2\ \mbox{for a.a.}\ z\in\Omega,\ \mbox{all}\ |x|\geq M_1.
	\end{equation}
	
	Let $t>0$. We have
	\begin{eqnarray}\label{eq10}
		\int_{\Omega}F(z,t\hat{u}_1(p))dz&=&\int_{\{t\hat{u}_1(p)\geq M_1\}}F(z,t\hat{u}_1(p))dz+\int_{\{0\leq t\hat{u}_1(p)<M_1\}}F(z,t\hat{u}_1(p))dz\nonumber\\
		&\geq&\xi t^2\int_{\{t\hat{u}_1(p)\geq M_1\}}\hat{u}_1(p)^2dz+\int_{\{0\leq t\hat{u}_1(p)<M_1\}}F(z,t\hat{u}_1(p))dz\ \mbox{(see (\ref{eq9}))}\nonumber\\
		&\geq&\xi t^2||\hat{u}_1(p)||^{2}_{2}-(\xi+c_1)t^2|\{0\leq t\hat{u}_1(p)<M_1\}|_N.
	\end{eqnarray}
	
	Note that $|\{0\leq t\hat{u}_1(p)<M_1\}|_N\rightarrow 0$ as $t\rightarrow\infty$ (recall that $\hat{u}_1(p)\in \mbox{int}\, C_+$). Also, $\xi>0$ is arbitrary. So, we see that for all large $t>0$, we have
	\begin{equation}\label{eq11}
		\xi t^2||\hat{u}_1(p)||^{2}_{2}-(\xi+c_1)t^2\left|\{0\leq t\hat{u}_1(p)<M_1\}\right|_N\geq-(m_1-1)+\frac{t^2}{2}||D\hat{u}_1(p)||^{2}_{2}.
	\end{equation}
	
	From (\ref{eq10}) and (\ref{eq11}) and for $t_0^1>0$ big, we have
	\begin{equation}\label{eq12}
		\int_{\Omega}F(z,t\hat{u}_1(p))dz\geq-(m_1-1)+\frac{t^2}{2}||D\hat{u}_1(p)||^{2}_{2}\ \mbox{for all}\ t\geq t_0^1.
	\end{equation}
	
	So, we have
	\begin{eqnarray*}
		 \varphi_{\lambda}(t_0^1\hat{u}_1(p))&=&\frac{(t_0^1)^p}{p}||D\hat{u}_1(p)||^{p}_{p}+\frac{(t_0^1)^2}{2}||D\hat{u}_1(p)||^{2}_{2}-\frac{\lambda(t_0^1)^p}{p}||\hat{u}_1(p)||^{p}_{p}-\\
		&&\hspace{7cm}\int_{\Omega}F(z,t_0^1\hat{u}_1(p))dz\\
		 &\leq&\frac{(t_0^1)^p[\hat{\lambda}_1(p)-\lambda]}{p}+\frac{(t_0^1)^2}{2}||D\hat{u}_1(p)||^{2}_{2}+m_1-1-\frac{(t_0^1)^2}{2}||D\hat{u}_1(p)||^{2}_{2}\\
		&&\hspace{4cm}(\mbox{see (\ref{eq12}) and recall that}\ ||\hat{u}_1(p)||_p=1)\\
		&\leq&\frac{(t_0^1)^p\epsilon}{p}+m_1-1\ \mbox{with}\ \epsilon>0\ (\mbox{recall}\ \lambda<\hat{\lambda}_1(p))\\
		&<&m_1\ \left(\mbox{by choosing}\ \epsilon>0\ \mbox{small such that}\ t_0^1<\left(\frac{p}{\epsilon}\right)^{1/p}\right).
	\end{eqnarray*}
	
	In a similar fashion, we can find large $t_0^2>0$ such that
	$$\varphi_{\lambda}(-t_0\hat{u}_1(p))<0\ \mbox{for all}\ \lambda\in(\hat{\lambda}_1(p)-\epsilon,\hat{\lambda}_1(p)),\ \mbox{all}\ t\geq t_0^2.$$
	
	Let $t_0=\max\{t_0^1,t_0^2\}$. Then
	$$\varphi_{\lambda}(\pm t_0\hat{u}_1(p))<m_1\ \mbox{for all}\ \lambda\in(\hat{\lambda}_1(p)-\epsilon, \hat{\lambda}_1(p))\ \mbox{with small}\ \epsilon>0\,.$$
\end{proof}

We introduce the following sets
\begin{eqnarray*}
	&&U_+=\{u\in W^{1,p}_{0}(\Omega):u=t\hat{u}_1(p)+v,\ t>0,\ v\in V\},\\
	&&U_-=\{u\in W^{1,p}_{0}(\Omega):u=-t\hat{u}_1(p)+v,\ t>0,\ v\in V\}.
\end{eqnarray*}
\begin{prop}\label{prop12}
	If hypothesis $H_1$ holds and $\lambda\in(\hat{\lambda}_1(p)-\epsilon,\hat{\lambda}_1(p))$ with $\epsilon>0$ as in Proposition \ref{prop11}, then problem \eqref{eqP} has at least two nontrivial solutions
	$$\hat{u}_+\in U_+\ \mbox{and}\ \hat{u}_-\in U_-$$
and	both are local minimizers of the energy functional $\varphi_{\lambda}$.
\end{prop}
\begin{proof}
	We introduce the functional
	$$\hat{\varphi}^+_{\lambda}(u)=\left\{
	\begin{array}{ll}
		\varphi_{\lambda}(u)&\mbox{if}\ u\in\overline{U}_+\\
		+\infty&\mbox{if}\ u\notin\overline{U}_+.
	\end{array}\right.$$
	
	Evidently, $\hat{\varphi}^+_{\lambda}$ is lower semicontinuous and bounded from below (see Proposition \ref{prop8}). So, we can apply the Ekeland variational principle (see, for example, Gasinski and Papageorgiou \cite[p. 582]{13}) and $\{u_n\}_{n\geq 1}\subseteq U_+$ such that
	\begin{eqnarray}
		&&\varphi_{\lambda}(u_n)=\hat{\varphi}^+_{\lambda}(u_n)\downarrow \inf \hat{\varphi}^+_{\lambda}\ \mbox{as}\ n\rightarrow\infty\label{eq13}\\
		 &&\varphi_{\lambda}(u_n)=\hat{\varphi}^+_{\lambda}(u_n)\leq\hat{\varphi}^+_{\lambda}(y)+\frac{1}{n(1+||u_n||)}||y-u_n||\label{eq14}\\
		&&\hspace{3cm}\mbox{for all}\ y\in W^{1,p}_{0}(\Omega),\ \mbox{all}\ n\geq 1.\nonumber
	\end{eqnarray}
	Fix $n\geq 1$ and let $h\in W^{1,p}_{0}(\Omega)$. Then for small $t>0$ we have $u_n+th\in U_+$. Using this as a test function in (\ref{eq14}), we have
	\begin{eqnarray}\label{eq15}
		 &&-\frac{||h||}{n(1+||u_n||)}\leq\frac{\varphi_{\lambda}(u_n+th)-\varphi_{\lambda}(u_n)}{t}\ (\mbox{note that}\ \left.\varphi_{\lambda}\right|_{\overline{U}_1}=\hat{\varphi}^+_{\lambda}\left.\right|_{\overline{U}_+})\nonumber\\
		&\Rightarrow&-\frac{||h||}{n(1+||u_n||)}\leq\left\langle \varphi'_{\lambda}(u_n),h\right\rangle\ (\mbox{recall}\ \varphi_{\lambda}\in C^1(W^{1,p}_{0}(\Omega))).
	\end{eqnarray}
	
	Since $h\in W^{1,p}_{0}(\Omega)$ is arbitrary, from (\ref{eq15}) it follows that
	$$(1+||u_n||)\varphi'_{\lambda}(u_n)\rightarrow 0\ \mbox{in}\ W^{-1,p'}(\Omega)\ \mbox{as}\ n\rightarrow\infty.$$
	
	But $\varphi_{\lambda}$ being coercive, satisfies the $C$-condition (see \cite{27}). So, it follows that
	$$u_n\rightarrow\hat{u}_+\ \mbox{in}\ W^{1,p}_{0}(\Omega)\ \mbox{as}\ n\rightarrow\infty.$$
	
	We have $\hat{u}_+\in\overline{U}_+$ and so from (\ref{eq13}) we infer that
	$$\varphi_{\lambda}(\hat{u}_+)=\inf\limits_{\overline{U}_+} \varphi_{\lambda}$$
	
	Suppose that $\hat{u}_+\in\partial U_+=V$. Then
	$$m_1\leq\inf\limits_{\overline{U}_+} \varphi_{\lambda}=\varphi_{\lambda}(\hat{u}_+)\ \mbox{(see (\ref{eq8}))},$$
	which contradicts Proposition \ref{prop11}. Therefore $\hat{u}_+\in U_+$ and it is a local minimizer of $\varphi_{\lambda}$, hence a nontrivial solution of \eqref{eqP}. By Ladyzhenskaya and Uraltseva \cite[p. 286]{15} we have $\hat{u}_+\in L^{\infty}(\Omega)$. Then we can apply Theorem 1 of Lieberman \cite{16} and obtain that $\hat{u}_+\in C^{1}_{0}(\overline{\Omega})$.
	
	Similarly, working with the functional
	$$\hat{\varphi}_{\lambda}(u)=\left\{
	\begin{array}{ll}
		\varphi_{\lambda}(u)&\mbox{if}\ u\in\overline{U}_-\\
		+\infty&\mbox{if}\ u\notin\overline{U}_-,
	\end{array}\right.$$
	we obtain a second nontrivial solution $\hat{u}_-\in U_-\cap C^{1}_{0}(\overline{\Omega})$, which is a local minimizer of $\varphi_{\lambda}$ and is distinct from $\hat{u}_+$.
\end{proof}

Next, using Morse theory, we will produce the third nontrivial solution. To this end, we need to compute the critical groups of $\varphi_{\lambda}$ at the origin.
\begin{prop}\label{prop13}
	If hypotheses $H_1$ hold and $\lambda>0$, then $C_k(\varphi_{\lambda},0)=\delta_{k,d_m}\ZZ$ for all $k\geq 0$ with $d_m=\mbox{dim}\,\overset{m}{\underset{\mathrm{i=1}}\oplus}\ E(\hat{\lambda}_i(2))\geq 2.$
\end{prop}
\begin{proof}
	Let $\psi:W^{1,p}_{0}(\Omega)\rightarrow\RR$ be the $C^2$-functional defined by
	 $$\psi(u)=\frac{1}{p}||Du||^{p}_{p}+\frac{1}{2}||Du||^{2}_{2}-\frac{1}{2}\int_{\Omega}\eta(z)u(z)^2dz\ \mbox{for all}\ u\in W^{1,p}_{0}(\Omega).$$
	
	We consider the homotopy
	$$h_{\lambda}(t,u)=(1-t)\varphi_{\lambda}(u)+t\psi(u)\ \mbox{for all}\ (t,u)\in[0,1]\times W^{1,p}_{0}(\Omega).$$
	
	Suppose that we can find $\{t_n\}_{n\geq 1}\subseteq[0,1]$ and $\{u_n\}_{n\geq 1}\subseteq W^{1,p}_{0}(\Omega)$ such that
	\begin{eqnarray}\label{eq16}
		&&t_n\rightarrow t\ \mbox{in}\ [0,1],\ u_n\rightarrow 0\ \mbox{in}\ W^{1,p}_{0}(\Omega)\ \mbox{as}\ n\rightarrow\infty\ \mbox{and}\ (h_{\lambda})'_u(t_n,u_n)=0\\
		&&\hspace{9cm}\mbox{for all}\ n\geq 1.\nonumber
	\end{eqnarray}
	
	We have
	\begin{equation}\label{eq17}
		A_p(u_n)+A(u_n)=(1-t_n)\lambda|u_n|^{p-2}u_n+(1-t_n)N_f(u_n)+t_n\eta u_n\ \mbox{for all}\ n\geq 1.
	\end{equation}
	
	Let $y_n=\frac{u_n}{||u_n||}\ n\geq 1$. Then $||y_n||=1$ for all $n\geq 1$ and so may assume that
	\begin{equation}\label{eq18}
		y_n\stackrel{w}{\rightarrow} y\ \mbox{in}\ W^{1,p}_{0}(\Omega)\ \mbox{and}\ y_n\rightarrow y\ \mbox{in}\ L^p(\Omega)\ \mbox{as}\ n\rightarrow\infty.
	\end{equation}
	
	From (\ref{eq17}), we have
	\begin{eqnarray}\label{eq19}
		 &&||u_n||^{p-2}A_p(y_n)+A(y_n)=(1-t_n)\lambda|u_n|^{p-2}y_n+(1-t_n)\frac{N_f(u_n)}{||u_n||}+t_n\eta y_n\\
		&&\hspace{8cm}\mbox{for all}\ n\geq 1.\nonumber
	\end{eqnarray}
	
	Note that hypothesis $H_1(i)$ and (\ref{eq16}), imply that $\left\{\frac{N_f(u_n)}{||u_n||}\right\}_{n\geq 1}\subseteq L^2(\Omega)$ is bounded. This fact, in conjunction with hypothesis $H_1(iii)$ implies (at least for a subsequence) that
	\begin{equation}\label{eq20}
		\frac{N_f(u_n)}{||u_n||}\ \stackrel{w}{\rightarrow}\ \eta y\ \mbox{in}\ L^2(\Omega)\ \mbox{as}\ n\rightarrow \infty\ (\mbox{see \cite{1}}).
	\end{equation}
	
	Also, we have that $\{A_p(y_n)\}_{n\geq 1}\subseteq W^{-1,p'}(\Omega)$ is bounded (see (\ref{eq18}) and Proposition \ref{prop6}). Therefore
	\begin{equation}\label{eq21}
		||u_n||^{p-2}A_p(y_n)\rightarrow 0\ \mbox{in}\ W^{-1,p'}(\Omega)\ \mbox{as}\ n\rightarrow\infty\ (\mbox{see (\ref{eq16})}).
	\end{equation}
	
	So, if in (\ref{eq19}) we pass to the limit as $n\rightarrow\infty$ and use (\ref{eq18}), (\ref{eq20}), (\ref{eq21}), then
	\begin{eqnarray}\label{eq22}
		&&A(y)=\eta y,\nonumber\\
		&\Rightarrow&-\Delta y(z)=\eta(z)y(z)\ \mbox{for a.a.}\ z\in\Omega,\ y\left.\right|_{\partial\Omega}=0.
	\end{eqnarray}
	
	From hypothesis $H_1(iii)$ and (\ref{eq22}) it follows that $y\equiv 0$. On the other hand, from (\ref{eq19}) we have
	\begin{equation}\label{eq23}
		\left\{\begin{array}{ll}
			||u_n||^{p-2}(-\Delta_py_n(z))-\Delta y_n(z)=&(1-t_n)\lambda|y_n(z)|^{p-2}y_n(z)+(1-t_n)\frac{f(z,u_n(z))}{||u_n||}\\
			&+t_n\eta(z)y_n(z)\ \mbox{for a.a.}\ z\in\Omega,\\
			\left.u_n\right|_{\partial\Omega}=0&
		\end{array}\right\}
	\end{equation}
	
	Then by (\ref{eq23}) and Ladyzhenskaya and Uraltseva \cite[p. 286]{15}, we know that we can find $M_2>0$ such that
	\begin{equation}\label{eq24}
		||u_n||_{\infty}\leq M_2\ \mbox{for all}\ n\geq 1.
	\end{equation}
	
	Since $||u_n||^{p-2}\rightarrow 0$ as $n\rightarrow\infty$ (see (\ref{eq16})), from (\ref{eq23}), (\ref{eq24}) and Theorem 1 of Lieberman \cite{16}, we know that there exist $\alpha\in(0,1)$ and $M_3>0$ such that
	$$y_n\in C^{1,\alpha}_{0}(\overline{\Omega})\ \mbox{and}\ ||y_n||_{C^{1,\alpha}_{0}(\overline{\Omega})}\leq M_3\ \mbox{for all}\ n\geq 1.$$
	
	Exploiting the compact embedding of $C^{1,\alpha}_{0}(\overline{\Omega})$ into $C^{1}_{0}(\overline{\Omega})$ and using (\ref{eq18}), we have
	\begin{eqnarray*}
		&&y_n\rightarrow y=0\ \mbox{in}\ C^{1}_{0}(\overline{\Omega})\ \mbox{as}\ n\rightarrow\infty,\\
		&\Rightarrow&y_n\rightarrow y=0\ \mbox{in}\ W^{1,p}_{0}(\Omega)\ \mbox{as}\ n\rightarrow\infty,
	\end{eqnarray*}
	which contradicts the fact that $||y_n||=1$ for all $n\geq 1$. Hence (\ref{eq16}) cannot occur and so by the homotopy invariance of critical groups we have
	\begin{equation}\label{eq25}
		C_k(\varphi_{\lambda},0)=C_k(\psi,0)\ \mbox{for all}\ k\geq 0.
	\end{equation}
	
	From Cingolani and Vannella \cite[Theorem 1.1]{9}  we know that
	\begin{eqnarray*}
		&&C_k(\psi,0)=\delta_{k,d_m}\ZZ\ \mbox{for all}\ k\geq 0,\\
		&\Rightarrow&C_k(\varphi_{\lambda},0)=\delta_{k,d_m}\ZZ\ \mbox{for all}\ k\geq 0\ \mbox{(see (\ref{eq25}))}.
	\end{eqnarray*}
\end{proof}
Now we can generate the third nontrivial solution.
\begin{prop}\label{prop14}
	If hypotheses $H_1$ hold and $\lambda\in(\hat{\lambda}_1(p)-\epsilon,\hat{\lambda}_1(p))$ with $\epsilon>0$ as in Proposition \ref{prop11}, then problem \eqref{eqP} admits a third nontrivial solution $\hat{y}\in C^{1}_{0}(\overline{\Omega})$.
\end{prop}
\begin{proof}
	Without any loss of generality, we may assume that $\varphi_{\lambda}(\hat{u}_-)\leq\varphi_{\lambda}(\hat{u}_+)$ (the analysis is similar if the opposite inequality holds). Also, we assume that $K_{\varphi_{\lambda}}$ is finite (otherwise we already have infinitely many solutions for problem \eqref{eqP}). From Proposition \ref{prop12}, we know that $\hat{u}_+\in C^{1}_{0}(\overline{\Omega})$ is a local minimizer of $\varphi_{\lambda}$. So, we can find small $\rho\in(0,1)$ such that
	\begin{eqnarray}\label{eq26}		 \varphi_{\lambda}(\hat{u}_-)\leq\varphi_{\lambda}(\hat{u}_+)<\inf[\varphi_{\lambda}(u):||u-\hat{u}_+||=\rho]=m^{\lambda}_{\rho},\ ||\hat{u}_--\hat{u}_+||>\rho
	\end{eqnarray}
	(see Aizicovici, Papageorgiou and Staicu \cite{1}, proof of Proposition 29). Recall that $\varphi_{\lambda}$ satisfies the $C$-condition. This fact and (\ref{eq26}) permit the use of Theorem \ref{th2} (the mountain pass theorem). So, we can find $\hat{y}\in W^{1,p}_{0}(\Omega)$ such that
	\begin{equation}\label{eq27}
		\hat{y}\in K_{\varphi_{\lambda}}\ \mbox{and}\ m^{\lambda}_{\rho}\leq\varphi_{\lambda}(\hat{y}).
	\end{equation}
	
	From (\ref{eq27}) it follows that $\hat{y}$ is a solution of \eqref{eqP} and $\hat{y}\notin\{\hat{u}_-,\hat{u}_+\}$. Since $\hat{y}$ is a critical point of $\varphi_{\lambda}$ of mountain pass type, we have
	\begin{equation}\label{eq28}
		C_1(\varphi_{\lambda},\hat{y})\neq 0.
	\end{equation}
	
	On the other hand, from Proposition \ref{prop13}, we have
	\begin{equation}\label{eq29}
		C_k(\varphi_{\lambda},0)=\delta_{k,d_m}\ZZ\ \mbox{for all}\ k\geq 0\ \mbox{with}\ d_m\geq 2.
	\end{equation}
	
	Comparing (\ref{eq28}) and (\ref{eq29}), we see that $\hat{y}\neq 0$. Nonlinear regularity theory (see \cite{16}) implies $\hat{y}\in C^{1}_{0}(\overline{\Omega})$. This is the third nontrivial solution of \eqref{eqP}.
\end{proof}

So, we can state our first multiplicity theorem for problem \eqref{eqP}.
\begin{theorem}\label{th15}
	If hypotheses $H_1$ hold, then there exists $\epsilon>0$ such that for all $\lambda\in(\hat{\lambda}_1(p)-\epsilon,\hat{\lambda}_1(p))$ problem \eqref{eqP} admits at least three nontrivial solutions
	$$\hat{u}_+,\hat{u}_-,\hat{y}\in C^{1}_{0}(\overline{\Omega}),$$
	with $\hat{u}_+$ and $\hat{u}_-$ being local minimizers of the energy functional $\varphi_{\lambda}$.
\end{theorem}

By strengthening the regularity conditions on $f(z,\cdot)$, we can improve Theorem \ref{th15} and produce the fourth nontrivial solution. The new hypotheses on $f(z,x)$ are the following:

\smallskip
$H_2:$ $f:\Omega\times\RR\rightarrow\RR$ is a measurable function such that for a.a. $z\in\Omega$, $f(z,0)=0,\ f(z,\cdot)\in C^1(\RR)$ and
\begin{itemize}
	\item[(i)] for every $\rho>0$, there exists $a_{\rho}\in L^{\infty}(\Omega)_+$ such that
	$$|f(z,x)|\leq a_{\rho}(z)\ \mbox{for a.a.}\ z\in\Omega,\ \mbox{all}\ |x|\leq\rho;$$
	\item[(ii)] $\lim\limits_{x\rightarrow\pm\infty}\ \frac{f(z,x)}{|x|^{p-2}x}=0$ uniformly for a.a. $z\in\Omega$ and if $F(z,x)=\int^{x}_{0}f(z,s)ds$, then
	$$\lim\limits_{x\rightarrow\pm\infty}\ \frac{F(z,x)}{x^2}=+\infty\ \mbox{uniformly for a.a.}\ z\in\Omega;\ \mbox{and}$$
	\item[(iii)] there exists an integer $m\geq 2$ such that
	\begin{eqnarray*}
		&&f'_x(z,0)\in[\hat{\lambda}_{m}(2),\hat{\lambda}_{m+1}(2)]\ \mbox{for a.a.}\ z\in\Omega,\ f'_x(\cdot,0)\not\equiv\hat{\lambda}_m(2),\ f'_x(\cdot,0)\not\equiv\hat{\lambda}_{m+1}(2)\\
		&&f'_x(z,0)=\lim\limits_{x\rightarrow 0}\ \frac{f(z,x)}{x}\ \mbox{uniformly for a.a.}\ z\in\Omega.
	\end{eqnarray*}
\end{itemize}

\begin{theorem}\label{th16}
	If hypotheses $H_2$ hold, then there exists $\epsilon>0$ such that for every $\lambda\in(\hat{\lambda}_1(p)-\epsilon,\hat{\lambda}_1(p))$ problem \eqref{eqP} has at least four nontrivial solutions
	$$\hat{u}_+,\hat{u}_-,\hat{y},\tilde{y}\in C^{1}_{0}(\overline{\Omega})$$
	with $\hat{u}_+$ and $\hat{u}_-$ being local minimizers of the energy functional $\varphi_{\lambda}$.
\end{theorem}
\begin{proof}
	From Theorem \ref{th15}, we already have three nontrivial solutions
	$$\hat{u}_+,\hat{u}_-,\hat{y}\in C^{1}_{0}(\overline{\Omega}),$$
	with $\hat{u}_+$ and $\hat{u}_-$ being local minimizers of $\varphi_{\lambda}$. Hence
	\begin{equation}\label{eq30}
		C_k(\varphi_{\lambda},\hat{u}_+)=C_k(\varphi_{\lambda},\hat{u}_-)=\delta_{k,0}\ZZ\ \mbox{for all}\ k\geq 0.
	\end{equation}
	
	Recall that
	\begin{equation}\label{eq31}
		C_1(\varphi_{\lambda},\hat{y})\neq 0\ (\mbox{see (\ref{eq28})}).
	\end{equation}
	
	Since $\varphi_{\lambda}\in C^2(W^{1,p}_{0}(\Omega))$, from (\ref{eq31}) and Papageorgiou and Smyrlis \cite{26} (see also Papageorgiou and R\u{a}dulescu \cite{23}) it follows that
	\begin{equation}\label{eq32}
		C_k(\varphi_{\lambda},\hat{y})=\delta_{k,1}\ZZ\ \mbox{for all}\ k\geq 0.
	\end{equation}
	
	From Theorem \ref{th15}, we know that
	\begin{equation}\label{eq33}
		C_k(\varphi_{\lambda},0)=\delta_{k,d_m}\ZZ\ \mbox{for all}\ k\geq 0.
	\end{equation}
	
	From Proposition \ref{prop8}, we know that $\varphi_{\lambda}$ is coercive. Therefore
	\begin{equation}\label{eq34}
		C_k(\varphi_{\lambda},\infty)=\delta_{k,0}\ZZ\ \mbox{for all}\ k\geq 0.
	\end{equation}
	
	Suppose that $K_{\varphi_{\lambda}}=\{0,\hat{u}_+,\hat{u}_-,\hat{y}\}$. Then from (\ref{eq30}), (\ref{eq32}), (\ref{eq33}), (\ref{eq34}) and the Morse relation (see (\ref{eq4})) with $t=-1$, we have
	\begin{eqnarray*}
		&&(-1)^{d_m}+2(-1)^0+(-1)^1=(-1)^0,\\
		&\Rightarrow&(-1)^{d_m}=0,\ \mbox{a contradiction}.
	\end{eqnarray*}
	
	So we can find $\tilde{y}\in K_{\varphi_{\lambda}}$, $\tilde{y}\notin\{0,\hat{u}_+,\hat{u}_-,\hat{y}\}$. It follows that $\tilde{y}$ is the fourth nontrivial solution of \eqref{eqP} and the nonlinear regularity theory implies $\tilde{y}\in C^{1}_{0}(\overline{\Omega})$.\end{proof}

\section{Near Resonance from the Right of $\hat{\lambda}_1(p)>0$}\label{sec4}

In this section we examine problem \eqref{eqP} as the parameter $\lambda$ approaches $\hat{\lambda}_1(p)>0$ from the above (from the right). In contrast to the previous case (Section \ref{sec3}), now the energy functional is indefinite.

We start with an existence result which is valid for all $\lambda$ in the open spectral interval $(\hat{\lambda}_1(p),\hat{\lambda}_2(p))$. The hypotheses on the perturbation $f(z,x)$ are the following:

\smallskip
$H_3:$ $f:\Omega\times\RR\rightarrow\RR$ is a Carath\'eodory function such that $f(z,0)=0$ for a.a. $z\in\Omega$ and
\begin{itemize}
	\item[(i)] for every $\rho>0$, there exists $a_{\rho}\in L^{\infty}(\Omega)_+$ such that
	$$|f(z,x)|\leq a_{\rho}(z)\ \mbox{for a.a.}\ z\in\Omega,\ \mbox{all}\ |x|\leq\rho;$$
	\item[(ii)] $\lim\limits_{x\rightarrow\pm\infty}\ \frac{f(z,x)}{|x|^{p-2}x}=0$ uniformly for a.a. $z\in\Omega$;
	\item[(iii)] if $F(z,x)=\int^{x}_{0}f(z,s)ds$, then there exists $\tau\in(2,p)$ and $\beta_0>0$ such that
	$$\beta_0\leq\liminf\limits_{x\rightarrow\pm\infty}\ \frac{pF(z,x)-f(z,x)x}{|x|^{\tau}}\ \mbox{uniformly for a.a.}\ z\in\Omega;\ \mbox{and}$$
	\item[(iv)] there exists a function $\vartheta\in L^{\infty}(\Omega)$ such that
	\begin{eqnarray*}
		&&\vartheta(z)\leq\hat{\lambda}_1(2)\ \mbox{for a.a.}\ z\in\Omega,\ \vartheta\not\equiv\hat{\lambda}_1(2)\\
		&&\limsup\limits_{x\rightarrow 0}\ \frac{2F(z,x)}{x^2}\leq\vartheta(z)\ \mbox{uniformly for a.a.}\ z\in\Omega.
	\end{eqnarray*}
\end{itemize}

As before, for every $\lambda>0$, $\varphi_{\lambda}:W^{1,p}_{0}(\Omega)\rightarrow\RR$ is the energy functional of problem \eqref{eqP} defined by
$$\varphi_{\lambda}(u)=\frac{1}{p}||Du||^{p}_{p}+\frac{1}{2}||Du||^{2}_{2}-\frac{\lambda}{p}||u||^{p}_{p}-\int_{\Omega}F(z,u(z))dz\ \mbox{for all}\ u\in W^{1,p}_{0}(\Omega).$$

We have $\varphi_{\lambda}\in C^1(W^{1,p}_{0}(\Omega))$.
\begin{prop}\label{prop17}
	If hypotheses $H_3$ hold and $\lambda>0$, then $\varphi_{\lambda}$ satisfies the $C$-condition.
\end{prop}
\begin{proof}
	Let $\{u_n\}_{n\geq 1}\subseteq W^{1,p}_{0}(\Omega)$ be a sequence such that
	\begin{eqnarray}
		&&|\varphi_{\lambda}(u_n)|\leq M_3\ \mbox{for some}\ M_3>0,\ \mbox{all}\ n\geq 1\label{eq35}\\
		&&(1+||u_n||)\varphi'_{\lambda}(u_n)\rightarrow 0\ \mbox{in}\ W^{-1,p'}(\Omega)\ \mbox{as}\ n\rightarrow\infty.\label{eq36}
	\end{eqnarray}
	
	From (\ref{eq36}) we have
	\begin{eqnarray}\label{eq37}
		&&|\left\langle \varphi'_{\lambda}(u_n),h\right\rangle|\leq\frac{\epsilon_n||h||}{1+||u_n||}\ \mbox{for all}\ h\in W^{1,p}_{0}(\Omega)\ \mbox{with}\ \epsilon_n\rightarrow 0^+,\nonumber\\
		&\Rightarrow&\left|\left\langle A_p(u_n),h\right\rangle+\left\langle A(u_n),h\right\rangle-\lambda\int_{\Omega}|u_n|^{p-2}u_nhdz-\int_{\Omega}f(z,u_n)hdz\right|\leq\frac{\epsilon_n||h||}{1+||u_n||}\\
		&&\hspace{10cm}\mbox{for all}\ n\geq 1.\nonumber
	\end{eqnarray}
	
	In (\ref{eq37}) we choose $h=u_n\in W^{1,p}_{0}(\Omega)$ and obtain
	\begin{equation}\label{eq38}
		 ||Du_n||^{p}_{p}+||Du_n||^{2}_{2}-\lambda||u_n||^{p}_{p}-\int_{\Omega}f(z,u_n)u_ndz\leq \epsilon_n\ \mbox{for all}\ n\geq 1.
	\end{equation}
	
	On the other hand from (\ref{eq35}), we have
	\begin{equation}\label{eq39}
		 -||Du_n||^{p}_{p}-\frac{p}{2}||Du_n||^{2}_{2}+\lambda||u_n||^{p}_{p}+\int_{\Omega}pF(z,u_n)dz\leq pM_3\ \mbox{for all}\ n\geq 1.
	\end{equation}
	
	We add (\ref{eq38}) and (\ref{eq39}). Then
	\begin{eqnarray}\label{eq40}
		&&\int_{\Omega}[pF(z,u_n)-f(z,u_n)u_n]dz\leq M_4+\left(\frac{p}{2}-1\right)||Du_n||^{2}_{2}\\
		&&\hspace{7cm}\mbox{for some}\ M_4>0,\ \mbox{all}\ n\geq 1.\nonumber
	\end{eqnarray}
	
	By virtue of hypotheses $H_3(i),(iii)$, we can find $\beta_1\in(0,\beta_0)$ and $c_3>0$ such that
	\begin{equation}\label{eq41}
		\beta_1|x|^{\tau}-c_3\leq pF(z,x)-f(z,x)x\ \mbox{for a.a.}\ z\in\Omega,\ \mbox{all}\ x\in\RR.
	\end{equation}
	
	We use (\ref{eq41}) in (\ref{eq40}) and obtain
	\begin{equation}\label{eq42}
		\beta_1||u_n||^{\tau}_{\tau}\leq M_5+\left(\frac{p}{2}-1\right)||Du_n||^{2}_{2}\ \mbox{for some}\ M_5>0\ \mbox{and all}\ n\geq 1.
	\end{equation}
	
	Suppose that $\{u_n\}_{n\geq 1}\subseteq W^{1,p}_{0}(\Omega)$ is unbounded. Then $||u_n||\rightarrow\infty$ as $n\rightarrow\infty$. Set $y_n=\frac{u_n}{||u_n||},\ n\geq 1$. By passing to a suitable subsequence if necessary, we may assume that
	\begin{equation}\label{eq43}
		y_n\stackrel{w}{\rightarrow}y\ \mbox{in}\ W^{1,p}_{0}(\Omega)\ \mbox{and}\ y_n\rightarrow y\ \mbox{in}\ L^p(\Omega)\ \mbox{as}\ n\rightarrow\infty.
	\end{equation}
	
	From (\ref{eq42}) we have
	\begin{eqnarray}\label{eq44}
		 &&\beta_1||y_n||^{\tau}_{\tau}\leq\frac{M_5}{||u_n||^{\tau}}+\left(\frac{p}{2}-1\right)\frac{1}{||u_n||^{\tau-2}}||Dy_n||^{2}_{2}\ \mbox{for all}\ n\geq 1,\nonumber\\
		&\Rightarrow&y_n\rightarrow 0\ \mbox{in}\ L^{\tau}(\Omega)\ \mbox{as}\ n\rightarrow\infty\ (\mbox{recall}\ 2<\tau<p),\ \mbox{hence}\ y=0\ (\mbox{see (\ref{eq43})}).
	\end{eqnarray}
	
	On the other hand, from (\ref{eq37}) we have
	\begin{eqnarray}\label{eq45}
		&&\left|\left\langle A_p(y_n),h\right\rangle+\frac{1}{||u_n||^{p-2}}\left\langle A(y_n),h\right\rangle-\lambda\int_{\Omega}|y_n|^{p-2}y_nhdz-\int_{\Omega}\frac{f(z,u_n)}{||u_n||^{p-1}}hdz\right|\leq\epsilon_n\\
		&&\hspace{10cm}\mbox{for all}\ n\geq 1.\nonumber
	\end{eqnarray}
	
	Hypotheses $H_3(i),(ii)$, imply that
	\begin{eqnarray*}
		&&|f(z,x)|\leq c_4(1+|x|^{p-1})\ \mbox{for a.a.}\ z\in\Omega,\ \mbox{all}\ x\in\RR,\ \mbox{some}\ c_4>0,\\
		&\Rightarrow&\left\{\frac{N_f(u_n)}{||u_n||^{p-1}}\right\}_{n\geq 1}\subseteq L^{p'}(\Omega)\ \mbox{is bounded}.
	\end{eqnarray*}
	
	If in (\ref{eq45}) we choose $h=y_n-y\in W^{1,p}_{0}(\Omega)$, pass to the limit as $n\rightarrow\infty$ and use (\ref{eq44}), we obtain
	\begin{eqnarray}\label{eq46}
		&&\lim\limits_{n\rightarrow\infty}\left\langle A_p(y_n),y_n-y\right\rangle=0\ (\mbox{recall}\ p>2),\nonumber\\
		&\Rightarrow&y_n\rightarrow y\ \mbox{in}\ W^{1,p}_{0}(\Omega)\ (\mbox{see Proposition \ref{prop6}}),\ \mbox{hence}\ ||y||=1.
	\end{eqnarray}
	
	Comparing (\ref{eq44}) and (\ref{eq46}), we reach a contradiction. This proves that $\{u_n\}_{n\geq 1}\subseteq W^{1,p}_{0}(\Omega)$ is bounded. So, we may assume that
	\begin{equation}\label{eq47}
		u_n\stackrel{w}{\rightarrow} u\ \mbox{in}\ W^{1,p}_{0}(\Omega)\ \mbox{and}\ u_n\rightarrow u\ \mbox{in}\ L^p(\Omega)\ \mbox{as}\ n\rightarrow\infty.
	\end{equation}
	
	In (\ref{eq37}) we choose $h=u_n-u\in W^{1,p}_{0}(\Omega)$, pass to the limit as $n\rightarrow\infty$ and use (\ref{eq47}). Then
	\begin{eqnarray*}
		&&\lim\limits_{n\rightarrow\infty}\left[\left\langle A_p(u_n),u_n-u\right\rangle+\left\langle A(u_n),u_n-u\right\rangle\right]=0,\\
		&\Rightarrow&\limsup\limits_{n\rightarrow\infty}\left[\left\langle A_p(u_n),u_n-u\right\rangle+\left\langle A(u),u_n-u\right\rangle\right]\leq 0\ (\mbox{since}\ A\ \mbox{is monotone}),\\
		&\Rightarrow&\limsup\limits_{n\rightarrow\infty}\left\langle A_p(u_n),u_n-u\right\rangle\leq 0\ (\mbox{see (\ref{eq47})}),\\
		&\Rightarrow&u_n\rightarrow u\ \mbox{in}\ W^{1,p}_{0}(\Omega)\ \mbox{as}\ n\rightarrow\infty.
	\end{eqnarray*}
	
	This proves that the functional $\varphi_{\lambda}$ satisfies the $C$-condition for all $\lambda>0$.
\end{proof}
\begin{prop}\label{prop18}
	If hypotheses $H_3$ hold and $\lambda>\hat{\lambda}_1(p)$, then $\varphi_{\lambda}(t\hat{u}_1(p))\rightarrow-\infty$ as $t\rightarrow\pm\infty$ (that is, $\left.\varphi_{\lambda}\right|_{\RR\hat{u}_1(p)}$ is anticoercive).
\end{prop}
\begin{proof}
	Hypothesis $H_3(ii)$ implies that
	\begin{equation}\label{eq48}
		\lim\limits_{x\rightarrow\pm\infty}\frac{F(z,x)}{|x|^p}=0\ \mbox{uniformly for a.a.}\ z\in\Omega.
	\end{equation}
	
	From (\ref{eq48}) and hypothesis $H_3(i)$, we see that given $\epsilon>0$, we can find $c_5=c_5(\epsilon)>0$ such that
	\begin{equation}\label{eq49}
		F(z,x)\geq-\epsilon|x|^p-c_5\ \mbox{for a.a.}\ z\in\Omega,\ \mbox{all}\ x\in\RR.
	\end{equation}
	
	Then for $t\neq 0$, we have
	\begin{eqnarray}\label{eq50}
		 \varphi_{\lambda}(t\hat{u}_1(p))&=&\frac{|t|^p}{p}\hat{\lambda}_1(p)+\frac{t^2}{2}||D\hat{u}_1(p)||^{2}_{2}-\frac{\lambda|t|^p}{p}-\int_{\Omega}F(z,t\hat{u}_1(p))dz\nonumber\\
		&&\hspace{6cm}(\mbox{recall}\ ||\hat{u}_1(p)||_p=1)\nonumber\\
		 &\leq&\frac{|t|^p}{p}[\hat{\lambda}_1(p)-\lambda]+\frac{t^2}{2}||D\hat{u}_1(p)||^{2}_{2}+\frac{\epsilon|t|^p}{p}+c_5|\Omega|_N\ (\mbox{see (\ref{eq49})})\nonumber\\
		 &=&\frac{|t|^p}{p}[\hat{\lambda}_1(p)+\epsilon-\lambda]+\frac{t^2}{2}||D\hat{u}_1(p)||^{2}_{2}+c_3|\Omega|_N.
	\end{eqnarray}
	
	Choose $\epsilon\in(0,\lambda-\hat{\lambda}_1(p))$ (recall $\lambda>\hat{\lambda}_1(p)$). Then from (\ref{eq50}) and since $p>2$, we have
	$$\varphi_{\lambda}(t\hat{u}_1(p))\rightarrow-\infty\ \mbox{as}\ t\rightarrow\pm\infty.$$
This completes the proof.
\end{proof}

Let $D=\{u\in W^{1,p}_{0}(\Omega):||Du||^{p}_{p}=\hat{\lambda}_2(p)||u||^{p}_{p}\}$.
\begin{prop}\label{prop19}
	If hypotheses $H_3$ hold and $\lambda\in(\hat{\lambda}_1(p),\hat{\lambda}_2(p))$, then $\left.\varphi_{\lambda}\right|_D$ is coercive.
\end{prop}
\begin{proof}
	From (\ref{eq48}) and hypothesis $H_3(i)$, we see that given $\epsilon>0$, we can find $c_6=c_6(\epsilon)>0$ such that
	\begin{equation}\label{eq51}
		F(z,x)\leq\frac{\epsilon}{p}|x|^p+c_6\ \mbox{for a.a.}\ z\in\Omega,\ \mbox{all}\ x\in\RR.
	\end{equation}
	
	Let $u\in D$. We have
	\begin{eqnarray}\label{eq52}
		 \varphi_{\lambda}(u)&=&\frac{1}{p}||Du||^{p}_{p}+\frac{1}{2}||Du||^{2}_{2}-\frac{\lambda}{p}||u||^{p}_{p}-\int_{\Omega}F(z,u)dz\nonumber\\
		 &\geq&\frac{1}{p}||Du||^{p}_{p}-\frac{\lambda}{p\hat{\lambda}_2(p)}||Du||^{p}_{p}-\frac{\epsilon}{p\hat{\lambda}_2(p)}||Du||^{p}_{p}-c_6|\Omega|_N\ (\mbox{see (\ref{eq51})})\nonumber\\
		 &=&\frac{1}{p}\left[1-\frac{\lambda+\epsilon}{\hat{\lambda}_2(p)}\right]||u||^p-c_6|\Omega|_N.
	\end{eqnarray}
	
	Choosing $\epsilon\in(0,\hat{\lambda}_2(p)-\lambda)$ (recall $\lambda\in(\hat{\lambda}_1(p),\hat{\lambda}_2(p))$), from (\ref{eq52}) we infer that $\left.\varphi_{\lambda}\right|_D$ is coercive.
 \end{proof}
By virtue of Proposition \ref{prop19}, we have
$$m_D=\inf\limits_{D} \varphi_{\lambda}>-\infty.$$

Then, invoking Proposition \ref{prop18}, we can find $t^*>0$ such that
\begin{equation}\label{eq53}
	\varphi_{\lambda}(\pm t^*\hat{u}_1(p))<m_D.
\end{equation}

We introduce the following sets
$$E_0=\{\pm t^*\hat{u}_1(p)\},\ E=\mbox{conv}\,\{\pm t^*\hat{u}_1(p)\}=\{-st^*\hat{u}_1(p)+(1-s)t^*\hat{u}_1(p):s\in[0,1]\}.$$

For this pair $\{E_0,E\}$ and the set $D$ introduced above, we have the following property.
\begin{prop}\label{prop20}
	The pair $\{E_0,E\}$ is linking with $D$ in $W^{1,p}_{0}(\Omega)$.
\end{prop}
\begin{proof}
	Let $\hat{G}=\{u\in W^{1,p}_{0}(\Omega):||Du||^{p}_{p}<\hat{\lambda}_2(p)||u||^{p}_{p}\}$. We claim that $-t^*\hat{u}_1(p)$ and $t^*\hat{u}_1(p)$ belong to different path components of the set $\hat{G}$. To this end, let $\gamma\in C([0,1],W^{1,p}_{0}(\Omega))$ be a path such that
	$$\gamma(0)=-t^*\hat{u}_1(p)\ \mbox{and}\ \gamma(1)=t^*\hat{u}_1(p).$$
	
	By virtue of Proposition \ref{prop4}, we have
	$$\hat{\lambda}_2(p)\leq \max\left[\frac{||D\gamma(t)||^{p}_{p}}{||\gamma(t)||^{p}_{p}}:t\in[0,1]\right]$$
	and so we can find $t_0\in(0,1)$ such that $\gamma(t_0)\notin\hat{G}$, which shows that $-t^*\hat{u}_1(p)$ and $t^*\hat{u}_1(p)$ cannot be in the same path component of the set $\hat{G}$. This means that, given any $\gamma\in C([0,1], W^{1,p}_{0}(\Omega))$ with
	$$\gamma(0)=-t^*\hat{u}_1(p)\ \mbox{and}\ \gamma(1)=t^*\hat{u}_1(p),$$
	we have
	$$\gamma([0,1])\cap\partial\hat{G}\neq\emptyset .$$
	
	Note that $\partial\hat{G}\subseteq D$. Therefore
	\begin{eqnarray*}
		&&\gamma([0,1])\cap D\neq\emptyset \\
		&\Rightarrow&\{E_0,E\}\ \mbox{links with}\ D\ \mbox{in}\ W^{1,p}_{0}(\Omega)\ (\mbox{see Definition \ref{def1}}).
	\end{eqnarray*}
\end{proof}
\begin{prop}\label{prop21}
	If hypothesis $H_3$ holds and $\lambda>0$, then $u=0$ is a local minimizer of the functional $\varphi_{\lambda}$.
\end{prop}
\begin{proof}
	By virtue of hypotheses $H_3(i),(iv)$ we see that given $\epsilon>0$, we can find $c_7=c_7(\epsilon)>0$ such that
	\begin{equation}\label{eq54}
		F(z,x)\leq\frac{1}{2}(\vartheta(z)+\epsilon)x^2+c_7|x|^p\ \mbox{for a.a.}\ z\in\Omega,\ \mbox{all}\ x\in\RR .
	\end{equation}
	
	Then for every $u\in W^{1,p}_{0}(\Omega)$, we have
	\begin{eqnarray*}
		 \varphi_{\lambda}(u)&\geq&\frac{1}{2}\left[||Du||^{2}_{2}-\int_{\Omega}\vartheta(z)u^2dz\right]-\frac{\epsilon}{2\hat{\lambda}_1(2)}||u||^2-c_8||u||^p-\frac{\lambda}{p\hat{\lambda}_1(p)}||u||^p\\
		&&\hspace{4.5cm}\mbox{for some}\ c_8>0\ \mbox{(see (\ref{eq1}), (\ref{eq2}) and (\ref{eq54}))}\\
		 &\geq&\frac{1}{2}\left[\hat{\xi}_0-\frac{\epsilon}{\hat{\lambda}_1(2)}\right]||u||^2-c_9||u||^p\ \mbox{for some}\ c_9>0\ (\mbox{see Proposition \ref{prop5}}).
	\end{eqnarray*}
	
	We choose $\epsilon\in(0,\hat{\lambda}_1(2)\hat{\xi}_0)$ and have
	\begin{equation}\label{eq55}
		\varphi_{\lambda}(u)\geq c_{10}||u||^2-c_9||u||^p\ \mbox{for some}\ c_{10}>0,\ \mbox{all}\ u\in W^{1,p}_{0}(\Omega).
	\end{equation}
	
	Since $2<p$, from (\ref{eq55}) it follows that we can find small $\rho\in(0,1)$ such that
	\begin{eqnarray*}
		&&\varphi_{\lambda}(u)>0=\varphi_{\lambda}(0)\ \mbox{for all}\ u\in W^{1,p}_{0}(\Omega)\ \mbox{with}\ 0<||u||\leq\rho,\\
		&\Rightarrow&u=0\ \mbox{is a (strict) local minimizer of}\ \varphi_{\lambda}.
	\end{eqnarray*}
\end{proof}
We can state the following existence result.
\begin{theorem}\label{th22}
	If hypothesis $H_3$ holds and $\lambda\in(\hat{\lambda}_1(p),\hat{\lambda}_2(p))$, then problem \eqref{eqP} admits a nontrivial solution $\hat{u}\in C^{1}_{0}(\overline{\Omega})$.
\end{theorem}
\begin{proof}
	Propositions \ref{prop17}, \ref{prop20}, and (\ref{eq53}), permit the use of Theorem 1 (the linking theorem). So, we can find $\hat{u}\in W^{1,p}_{0}(\Omega)$ such that
	\begin{equation}\label{eq56}
		\hat{u}\in K_{\varphi_{\lambda}}\ \mbox{and}\ C_1(\varphi_{\lambda},\hat{u})\neq0\ (\mbox{see Chang \cite{6}}).
	\end{equation}
	
	By Proposition \ref{prop21}, we know that $u=0$ is a local minimizer of $\varphi_{\lambda}$. Hence
	\begin{equation}\label{eq57}
		C_k(\varphi_{\lambda},0)=\delta_{k,0}\ZZ\ \mbox{for all}\ k\geq 0.
	\end{equation}
	
	From (\ref{eq56}) and (\ref{eq57}) it follows that $\hat{u}\neq 0$ and $\hat{u}$ is a solution of \eqref{eqP}. Moreover, the nonlinear regularity theory implies that $\hat{u}\in C^{1}_{0}(\overline{\Omega})$.
\end{proof}

We can have multiple solutions when we restrict $\lambda$ to be near $\hat{\lambda}_1(p)$ from above (near resonance from the right). To do this, we introduce the following hypotheses on the perturbation $f(z,x)$.

\smallskip
$H_4:$ $f:\Omega\times\RR\rightarrow\RR$ is a Carath\'eodory function such that $f(z,0)=0$ for a.a. $z\in\Omega$ and
\begin{itemize}
	\item[(i)] $|f(z,x)|\leq a(z)(1+|x|^{r-1})$ for a.a. $z\in\Omega$, all $x\in\RR$ with $a\in L^{\infty}(\Omega)_+$;
	\item[(ii)] there exists a function $\vartheta\in L^{\infty}(\Omega)$, $\vartheta(z)\leq 0$ for a.a. $z\in\Omega$, $\vartheta\not\equiv 0$ such that
	$$\limsup\limits_{x\rightarrow\pm\infty}\ \frac{pF(z,x)}{|x|^p}\leq\vartheta(z)\ \mbox{uniformly for a.a.}\ z\in\Omega;$$
	\item[(iii)] there exist an integer $m\geq 2$ and a function $\eta\in L^{\infty}(\Omega)_+$ such that
	\begin{eqnarray*}
		&&\eta(z)\in[\hat{\lambda}_m(2),\hat{\lambda}_{m+1}(2)]\ \mbox{for a.a.}\ z\in\Omega,\ \eta\not\equiv\hat{\lambda}_m(2),\ \eta\not\equiv\hat{\lambda}_{m+1}(2)\\
		&&\lim\limits_{x\rightarrow 0}\ \frac{f(z,x)}{x}=\eta(z)\ \mbox{uniformly for a.a.}\ z\in\Omega;\ \mbox{and}
	\end{eqnarray*}
	\item[(iv)] for every $\rho>0$ there exists $\xi_{\rho}>0$ such that for a.a. $z\in\Omega$ the function
	$$x\longmapsto f(z,x)+\xi_{\rho}|x|^{p-2}x$$
	is nondecreasing on $[-\rho, \rho]$.
\end{itemize}
\begin{remark}
	Evidently, for a.a. $z\in\Omega,\ f(z,\cdot)$ is differentiable at $x=0$ and $\eta(\cdot)=f'_x(\cdot,0)$.
\end{remark}

We will produce solutions of constant sign. For this purpose, we introduce the positive and negative truncations of $f(z,\cdot)$, namely the Carath\'eodory functions
$$f_{\pm}(z,x)=f(z,\pm x^{\pm}).$$

Let $F_{\pm}(z,x)=\int^{x}_{0}f_{\pm}(z,s)ds$ and consider the $C^1$-functionals $\varphi^{\pm}_{\lambda}:W^{1,p}_{0}(\Omega)\rightarrow\RR$ defined by
\begin{eqnarray*} &&\varphi^{\pm}_{\lambda}(u)=\frac{1}{p}||Du||^{p}_{p}+\frac{1}{2}||Du||^{2}_{2}-\frac{\lambda}{p}||u^{\pm}||^{p}_{p}-\int_{\Omega}F_{\pm}(z,u(z))dz\\
	&&\hspace{8cm}\mbox{for all}\ u\in W^{1,p}_{0}(\Omega).
\end{eqnarray*}

Next, we produce a pair of nontrivial constant sign solutions.
\begin{prop}\label{prop23}
	If hypothesis $H_4$ holds, then we can find $\epsilon>0$ such that for all $\lambda\in(\hat{\lambda}_1(p),$\\ $\hat{\lambda}_1(p)+\epsilon)$ problem \eqref{eqP} has at least two nontrivial solutions of constant sign
	$$u_n\in int\, C_+\ \mbox{and}\ v_0\in -int\, C_+,$$
	both being local minimizers of the energy functional $\varphi_{\lambda}.$
\end{prop}
\begin{proof}
	By virtue of hypotheses $H_4(i),(ii)$, given $\delta>0$, we can find $c_{11}=c_{11}(\delta)>0$ such that
	\begin{equation}\label{eq58}
		F(z,x)\leq\frac{1}{p}(\vartheta(z)+\delta)|x|^p+c_{11}\ \mbox{for a.a.}\ z\in\Omega,\ \mbox{all}\ x\in\RR.
	\end{equation}
	
	Since $\lambda>\hat{\lambda}_1(p)$, we have $\lambda=\hat{\lambda}_1(p)+\mu$ with $\mu>0$. Then for every $u\in W^{1,p}_{0}(\Omega)$ we have
	\begin{eqnarray*}
		 \varphi^{+}_{\lambda}(u)&=&\frac{1}{p}||Du||^{p}_{p}+\frac{1}{2}||Du||^{2}_{2}-\frac{\hat{\lambda}_1(p)}{p}||u^+||^{p}_{p}-\frac{\mu}{p}||u^+||^{p}_{p}-\int_{\Omega}F_+(z,u)dz\\
		 &\geq&\frac{1}{p}||Du||^{p}_{p}-\int_{\Omega}(\hat{\lambda}_1(p)+\vartheta(z))(u^+)^pdz-
\frac{\mu+\delta}{p\hat{\lambda}_1(p)}||u||^p-c_{11}|\Omega|_N\ (\mbox{see (\ref{eq58})})\\
		 &\geq&\frac{1}{p}\left[\xi^*-\frac{\mu+\delta}{\hat{\lambda}_1(p)}\right]||u||^p-c_{11}|\Omega|_N\ \mbox{for some}\ \xi^*>0\\
		&&\hspace{-2cm}(\mbox{see Papageorgiou and Kyritsi \cite[p. 356]{21}}).
	\end{eqnarray*}
	
	Since $\delta>0$, is arbitrary, for $\mu\in(0,\xi^*\ \hat{\lambda}_1(p))$, we have that $\varphi^{+}_{\lambda}$ is coercive. Also, using the Sobolev embedding theorem, we see that $\varphi^{+}_{\lambda}$ is sequentially weakly lower semicontinuous. So, by the Weierstrass theorem, we can find $u_0\in W^{1,p}_{0}(\Omega)$ such that
	\begin{equation}\label{eq59}
		\varphi^{+}_{\lambda}(u_0)=\inf[\varphi^{+}_{\lambda}(u):u\in W^{1,p}_{0}(\Omega)].
	\end{equation}
	
	Hypothesis $H_4(iii)$ implies that for small $t\in(0,1)$ 
	\begin{eqnarray*}
		&&\varphi^{+}_{\lambda}(t\hat{u}_1(2))<0\ (\mbox{recall that}\ p>2),\\
		&\Rightarrow&\varphi^{+}_{\lambda}(u_0)<0=\varphi^{+}_{\lambda}(0)\ (\mbox{see (\ref{eq59})}),\ \mbox{hence}\ u_0\neq 0.
	\end{eqnarray*}
	
	From (\ref{eq59}) we have
	\begin{eqnarray}\label{eq60}
		&&(\varphi^{+}_{\lambda})'(u_0)=0,\nonumber\\
		&\Rightarrow&A_p(u_0)+A(u_0)=\lambda(u_0^+)^{p-1}+N_{f_+}(u_0).
	\end{eqnarray}
	
	On (\ref{eq60}) we act with $-u_0^-\in W^{1,p}_{0}(\Omega)$ and obtain $u_0\geq 0,\ u_0\neq 0$. So, (\ref{eq60}) becomes
	\begin{eqnarray*}
		&&A_p(u_0)+A(u_0)=\lambda u_0^{p-1}+N_{f}(u_0),\\
		&\Rightarrow&u_0\ \mbox{is a solution of \eqref{eqP}},\ u_0\in C_+\backslash\{0\}\\
		&&\hspace{4cm}\mbox{(by the nonlinear regularity theory)}.
	\end{eqnarray*}
	
	Let $\rho=||u_n||_{\infty}$ and let $\xi_{\rho}>0$ be as postulated by hypothesis $H_2(iv)$. Then
	\begin{eqnarray*}
		&&-\Delta_pu_0(z)-\Delta u_0(z)+\xi_{\rho}u_0(z)^{p-1}\\
		&=&(\lambda+\xi_{\rho})u_0(z)^{p-1}+f(z,u_0(z))\geq 0\ \mbox{for a.a.}\ z\in\Omega,\\
		&\Rightarrow&\Delta_pu_0(z)+\Delta u_0(z)\leq\xi_{\rho}u_p(z)\ \mbox{for a.a.}\ z\in\Omega.
	\end{eqnarray*}
	
	From the nonlinear maximum principle of Pucci and Serrin \cite[pp. 111 and 120]{29}, we obtain that $u_0\in \mbox{int}\, C_+$. Since $\left.\varphi_{\lambda}\right|_{C_+}=\left.\varphi^{+}_{\lambda}\right|_{C_+}$, we infer that $u_0\in \mbox{int}\, C_+$ is a local $C^{1}_{0}(\overline{\Omega})$ minimizer of $\varphi_{\lambda}$. Invoking Proposition \ref{prop7}, we infer that $u_0$ is a local $W^{1,p}_{0}(\Omega)$-minimizer of $\varphi_{\lambda}$.
	
	Similarly, working with $\varphi^{-}_{\varphi}$ we produce $v_0\in-\mbox{int}\, C_+$ a second nontrivial constant sign solution of \eqref{eqP}, which is a local minimizers of $\varphi_{\lambda}$.
\end{proof}

Let $\epsilon>0$ be as in the above proposition. Hypotheses $H_4(i),(iii)$ imply that given $\delta>0$, we can find $c_{12}=c_{12}(\delta)>\hat{\lambda}_1(p)+\epsilon$ such that
\begin{equation}\label{eq61}
	f(z,x)x\geq(\eta(z)-\delta)x^2-c_{12}|x|^p\ \mbox{for a.a.}\ z\in\Omega,\ \mbox{all}\ x\in\RR.
\end{equation}

This estimate leads to the following auxiliary Dirichlet problem
\begin{equation}\label{eq62}
	-\Delta_pu(z)-\Delta u(z)=(\eta(z)-\delta )u(z)-c_{13}|u(z)|^{p-2}u(z)\ \mbox{in}\ \Omega,\ u|_{\partial\Omega}=0
\end{equation}
where $c_{13}=c_{13}(\delta,\lambda)=c_{12}-\lambda$, with $\lambda\in(\hat{\lambda}_1(p),\hat{\lambda}_1(p)+\epsilon)$.

\begin{prop}\label{prop24}
	For small $\delta>0$, problem (\ref{eq62}) has a unique nontrivial positive solution $u_*\in {\rm int}\, C_+$ and because (\ref{eq62}) is odd $v_*=-u_*\in-int\, C_+$ is the unique nontrivial negative solution of (\ref{eq62}).
\end{prop}
\begin{proof}
	First we establish the existence of a nontrivial positive solution. To this end, let $\psi_+:W^{1,p}_{0}(\Omega)\rightarrow\RR$ be the $C^1$-functional defined by
	\begin{eqnarray*}
		 &&\psi_+(u)=\frac{1}{p}||Du||^{p}_{p}+\frac{1}{2}||Du||^{2}_{2}-\frac{1}{2}\int_{\Omega}(\eta(z)+\delta)(u^+)^2dz+\frac{c_{13}}{p}||u^+||^{p}_{p}\\
		&&\hspace{6cm}\mbox{for all}\ u\in W^{1,p}_{0}(\Omega).
	\end{eqnarray*}
	
	Since $p>2$, it is clear that $\psi_+$ is coercive. Also, it is sequentially weakly lower semicontinuous. So, we can find $u_*\in W^{1,p}_{0}(\Omega)$ such that
	\begin{equation}\label{eq63}
		\psi_+(u_*)=\inf[\psi_+(u):u\in W^{1,p}_{0}(\Omega)].
	\end{equation}
	
	Let $t>0$. We have
	\begin{eqnarray*}
		 \psi_+(t\hat{u}_1(2))&=&\frac{t^p}{p}||D\hat{u}_1(2)||^{p}_{p}+\frac{t^2}{2}\hat{\lambda}_1(2)-\frac{t^2}{2}\int_{\Omega}(\eta(z)-\delta)\hat{u}_1(2)^2dz+\frac{c_{13}}{p}t^p||\hat{u}_1(2)||^{p}_{p}\\
		&&\hspace{6cm}(\mbox{recall}\ ||\hat{u}_1(2)||_2=1)\\
		 &\leq&\frac{t^p}{p}\left[1+\frac{c_{13}}{\hat{\lambda}_1(p)}\right]||\hat{u}_1(2)||^p-\frac{t^2}{2}\left[\int_{\Omega}(\eta(z)-\hat{\lambda}_1(2))\hat{u}_1(2)^2dz-\delta\right].
	\end{eqnarray*}
	
	Evidently, $\xi_0=\int_{\Omega}(\eta(z)-\hat{\lambda}_1(2))\hat{u}_1(2)^2dz>0$. So, if $\delta\in(0,\xi_0)$, then
	$$\psi_+(t\hat{u}_1(2))\leq\frac{t^p}{p}c_{14}-\frac{t^2}{2}c_{15}\ \mbox{some}\ c_{14},c_{15}>0.$$
	
	Since $p>2$, by choosing small $t\in(0,1)$, we have
	\begin{eqnarray*}
		&&\psi_+(t\hat{u}_1(2))<0,\\
		&\Rightarrow&\psi_+(u_*)<0=\psi_+(0)\ \mbox{(see (\ref{eq63})), hence}\ u_*\neq 0.
	\end{eqnarray*}
	
	From (\ref{eq63}) we have
	\begin{eqnarray}\label{eq64}
		&&\psi'_+(u_*)=0,\nonumber\\
		&\Rightarrow&A_p(u_*)+A(u_*)=(\eta-\delta)u^+_*-c_{13}(u^+_*)^{p-1}.
	\end{eqnarray}
	
	On (\ref{eq64}) we act with $-u^-_*\in W^{1,p}_{0}(\Omega)$ and obtain $u^*\geq 0,\ u_*\neq 0$. Then
	\begin{eqnarray*}
		&&A_p(u_*)+A(u_*)=(\eta-\delta)u_*-c_{13}u^{p-1}_{*},\\
		&\Rightarrow&u_*\in C_+\backslash\{0\}\ (\mbox{nonlinear regularity solves (\ref{eq62})}).
	\end{eqnarray*}
	
	In fact, we have
	\begin{eqnarray*}
		&&\Delta_pu_*(z)+\Delta u_*(z)\leq c_{13}u_*(z)^{p-1}\ \mbox{for a.a.}\ z\in\Omega\\
		&\Rightarrow&u_*\in \mbox{int}\, C_+\ (\mbox{see Pucci and Serrin \cite[pp. 111 and 120]{29}}).
	\end{eqnarray*}
	
	Next, we show the uniqueness of this positive solutions. To this end, let
	$$G_0(t)=\frac{t^p}{p}+\frac{t^2}{2}\ \mbox{for all}\ t\geq 0.$$
	
	Then $G_0(\cdot)$ is increasing and $t\rightarrow G_0(t^{1/2})$ is convex. We set
	$$G(y)=G_0(|y|)\ \mbox{for all}\ y\in\RR^N.$$
	
	Evidently, $G\in C^1(\RR^N)$ (recall $p>2$) and we have
  \begin{eqnarray*}
		&&\nabla G(y)=a(y)=|y|^{p-2}y+y\ \mbox{for all}\ y\in\RR^N\\
		&&\mbox{div}\, a(Du)=\Delta_p u+\Delta u\ \mbox{for all}\ u\in W^{1,p}_{0}(\Omega).
	\end{eqnarray*}
	
	Let $\mu_+:L^1(\Omega)\rightarrow\RR$ be the integral functional defined by
	$$\mu_+(u)=\left\{
	\begin{array}{ll}
		\displaystyle\int_{\Omega}G(Du^{1/2})dz&\mbox{if}\ u\geq 0,\ u^{1/2}\in W^{1,p}_{0}(\Omega)\\
		+\infty&\mbox{otherwise}.
	\end{array}\right.$$
	
	Let $u_1,u_2\in \mbox{dom}\, \mu_+=\{u\in L^1(\Omega):\mu_+(u)<\infty\}$ (the effective domain of $\mu_+$) and let $y=(tu_1+(1-t)u_2)^{1/2}\in W^{1,p}_{0}(\Omega)$ with $t\in[0,1]$. From Benguria, Brezis and Lieb \cite[Lemma 4]{5}, we have
	\begin{eqnarray*}
		&|Dy(z)|&\leq\left(t|Du_1(z)^{1/2}|^2+(1-t)|Du_2(z)^{1/2}|^2\right)^{1/2}\ \mbox{for a.a.}\ z\in\Omega,\\
		\Rightarrow&G_0(|Dy(z)|)&\leq G_0\left(t|Du_1(z)^{1/2}|^2+(1-t)|Du_2(z)^{1/2}|^2\right)\ (\mbox{since}\ G_0\ \mbox{is increasing)}\\
		&&\leq tG_0(|Du_1(z)^{1/2}|)+(1-t)G_0(|Du_2(z)^{1/2}|)\ (\mbox{since}\ t\rightarrow G_0(t^{1/2})\\
		&&\hspace{7.0cm}\mbox{is convex}),\\
		\Rightarrow&G(Dy(z))&\leq tG(Du_1(z)^{1/2})+(1-t)G(Du_2(z)^{1/2})\ \mbox{for a.a.}\ z\in\Omega,\\
		\Rightarrow&\mu_+&\mbox{is convex.}
	\end{eqnarray*}
	
	Also, by the Fatou lemma we see that $\mu_+$ is lower semicontinuous.
	
	Let $y_*\in W^{1,p}_{0}(\Omega)$ be another positive solution of (\ref{eq62}). From the first part of the proof, we have $y_*\in \mbox{int}\, C_+$. Let $h\in C^{1}_{0}(\overline{\Omega})$ and $t\in(-1,1)$ with $|t|$ small. Then we will have
	\begin{eqnarray*}
		&&u_*^2+th\in \mbox{int}\, C_+\ \mbox{and}\ y_*^2\ th\in \mbox{int}\, C_+\\
		&\Rightarrow&u_*^2,\ y_*^2\in \mbox{dom}\, \mu_+.
	\end{eqnarray*}
	
	So, $\mu_+$ is G\^ateaux differentiable at $u_*$ and at $y_*$ in the direction $h$. Using the chain rule, we obtain
	\begin{eqnarray*}
		&&\mu'_+(u_*^2)(h)=\frac{1}{2}\int_{\Omega}\frac{-\Delta_pu_*-\Delta u_*}{u_*}hdz\\
		&&\mu'_+(y_*^2)(h)=\frac{1}{2}\int_{\Omega}\frac{-\Delta_py_*-\Delta y_*}{y_*}hdz\ \mbox{for all}\ h\in C^{1}_{0}(\overline{\Omega}).
	\end{eqnarray*}
	
	The convexity of $\mu_+$ implies that $\mu'_+$ is monotone. Hence
	\begin{eqnarray*}
		0&\leq&\frac{1}{2}\int_{\Omega}\left(\frac{-\Delta_pu_*-\Delta u_*}{u_*}-\frac{-\Delta_py_*-\Delta y_*}{y_*}\right)(u^{2}_{*}-y^{2}_{*})dz\\
		&=&\frac{1}{2}\int_{\Omega}c_{13}(y^{p-2}_{*}-u^{p-2}_{*})(u^{2}_{*}-y^{2}_{*})dz\leq 0\ (\mbox{recall}\ p>2),\\
		\Rightarrow&&u_*=y_*.
	\end{eqnarray*}
	
	This proves the uniqueness of the positive solution $u_*\in \mbox{int}\, C_+$.
	
	Since (\ref{eq62}) is odd, $v_*=-u_*\in-\mbox{int}\, C_+$ is the unique nontrivial negative solution of (\ref{eq62}).
\end{proof}

Using the proposition, we can establish the existence of extremal nontrivial constant sign solutions, that is, a smallest positive solution and a biggest nontrivial negative solution.

\begin{prop}\label{prop25}
	If hypothesis $H_4$ holds and $\lambda\in(\hat{\lambda}_1(p),\hat{\lambda}_1(p)+\epsilon)$ with $\epsilon>0$ as in Proposition \ref{prop23}, then problem \eqref{eqP} admits a smallest positive solution $u^{*}_{\lambda}\in {\rm int}\, C_+$ and a biggest negative solution $v^{*}_{\lambda}\in-{\rm int}\, C_+$.
\end{prop}
\begin{proof}
	Let $S_+(\lambda)$ be the set of positive of problem \eqref{eqP}. From Proposition \ref{prop23} and its proof, we have
	$$S_+(\lambda)\neq\emptyset \ \mbox{and}\ S_+(\lambda)\subseteq \mbox{int}\, C_+.$$
	
	As in Gasinski and Papageorgiou \cite{14}, exploiting the monotonicity of $u\rightarrow A_p(u)+A(u)$ we have that the solution set $S_+(\lambda)$ is downward directed, that is, if $u_1,u_2\in S_+(\lambda)$, then we can find $u\in S_+(\lambda)$ such that $u\leq u_1,\ u\leq u_2$. Since we are looking for the smallest positive solution, without any loss of generality we may assume that there exists $M_6>0$ such that
	\begin{equation}\label{eq65}
		0\leq u(z)\leq M_6\ \mbox{for all}\ z\in\overline{\Omega},\ \mbox{all}\ u\in S_+(\lambda).
	\end{equation}
	
	From Dunford and Schwartz \cite[p. 336]{11}, we know that we can find $\{u_n\}_{n\geq 1}\subseteq S_+(\lambda)$ such that $\inf S_+(\lambda)=\inf\limits_{n\geq 1}\ u_n.$
	
	We have
	\begin{eqnarray}\label{eq66}
		&&A_p(u_n)+A(u_n)=\lambda u_n^{p-1}+N_f(u_n)\ \mbox{for all}\ n\geq 1,\\
		&\Rightarrow&\{u_n\}_{n\geq 1}\subseteq W^{1,p}_{0}(\Omega)\ \mbox{is bounded (see (\ref{eq65}))}.\nonumber
	\end{eqnarray}
	
	So, we may assume that
	\begin{equation}\label{eq67}
		u_n\stackrel{w}{\rightarrow}u^{*}_{\lambda}\ \mbox{in}\ W^{1,p}_{0}(\Omega)\ \mbox{and}\ u_n\rightarrow u^{*}_{\lambda}\ \mbox{in}\ L^p(\Omega)\ \mbox{as}\ n\rightarrow\infty.
	\end{equation}
	
	On (\ref{eq66}) we act with $u_n-u^{*}_{\lambda}\in W^{1,p}_{0}(\Omega)$, pass to the limit as $n\rightarrow\infty$ and use (\ref{eq67}). Then
	\begin{eqnarray}\label{eq68}
		&&\lim\limits_{n\rightarrow\infty}\left[\left\langle A_p(u_n),u_n-u^{*}_{\lambda}\right\rangle+\left\langle A(u_n),u_n-u^{*}_{\lambda}\right\rangle\right]=0,\nonumber\\
		&\Rightarrow&u_n\rightarrow u^{*}_{\lambda}\ \mbox{in}\ W^{1,p}_{0}(\Omega)\ \mbox{as}\ n\rightarrow\infty\ (\mbox{see the proof of Proposition \ref{prop17}}).
	\end{eqnarray}
	
	\begin{claim}
		$u_*\leq u$ for all $u\in S_+(\lambda)$.
	\end{claim}
	
		Let $u\in S_+(\lambda)$ and consider the following function
		\begin{eqnarray}\label{eq69}
			\beta_+(z,x)=\left\{
			\begin{array}{ll}
				0&\mbox{if}\ x<0\\
				(\eta(z)-\delta)x-c_{13}x^{p-1}&\mbox{if}\ 0\leq x\leq u(z)\\
				(\eta(z)-\delta)u(z)-c_{13}u(z)^{p-1}&\mbox{if}\ u(z)<x
			\end{array}\right.\ (\mbox{see (\ref{eq61})})
		\end{eqnarray}
		(see (\ref{eq61})). This is a Carath\'eodory function. We set $B_+(z,x)=\int^{x}_{0}\beta_+(z,s)ds$ and consider the $C^1$-functional $\xi_+:W^{1,p}_{0}(\Omega)\rightarrow\RR$ defined by
		 $$\xi_+(u)=\frac{1}{p}||Du||^{p}_{p}+\frac{1}{2}||Du||^{2}_{2}-\int_{\Omega}B_+(z,u(z))dz\ \mbox{for all}\ u\in W^{1,p}_{0}(\Omega).$$
		
		From (\ref{eq69}) it is clear that $\xi_+$ is coercive. Also, it is sequentially weakly lower semicontinuous. So, we can find $\hat{u}_*\in W^{1,p}_{0}(\Omega)$ such that
		\begin{equation}\label{eq70}
			\xi_+(\hat{u}_*)=\inf[\xi_+(u):u\in W^{1,p}_{0}(\Omega)].
		\end{equation}
			
		As before (see the proof of Proposition \ref{prop24}), for $y\in \mbox{int}\, C_+$, and for small $t>0$  (at least such that $ty\leq u$, recall that $u\in \mbox{int}\, C_+$, and see Lemma 3.3. of Filippakis, Kristaly and Papageorgiou \cite{12}), we have
		\begin{eqnarray*}
			&&\xi_+(ty)<0=\xi_+(0),\\
			&\Rightarrow&\xi_+(\hat{u}_*)<0=\xi_+(0)\ \mbox{(see (\ref{eq70})), hence}\ \hat{u}_*\neq 0.
		\end{eqnarray*}
		
		From (\ref{eq70}) we have
		\begin{eqnarray}\label{eq71}
			&&\xi_+'(\hat{u}_*)=0,\nonumber\\
			&\Rightarrow&A_p(\hat{u}_*)+A(\hat{u}_*)=N_{\beta_+}(\hat{u}_*).
		\end{eqnarray}
		
		On (\ref{eq71}) we act with $-\hat{u}^{-}_{*}\in W^{1,p}_{0}(\Omega)$ and obtain $\hat{u}_*\geq 0,\ \hat{u}_*\neq 0$ (see (\ref{eq69})). Also, on (\ref{eq71}) we act with $(\hat{u}_*-u)^+\in W^{1,p}_{0}(\Omega)$ we have
		\begin{eqnarray*}
			&&\left\langle A_p(\hat{u}_*),(\hat{u}_*-u)^+\right\rangle+\left\langle A(\hat{u}_*),(\hat{u}_*-u)^+\right\rangle\\
			&=&\int_{\Omega}\beta_+(z,\hat{u}_*)(\hat{u}_*-u)^+dz\\
			&=&\int_{\Omega}[(\eta(z)-\delta)u-c_{13}u^{p-1}](\hat{u}_*-u)^+ dz\ (\mbox{see (\ref{eq71})})\\
			&\leq&\int_{\Omega}[\lambda u^{p-1}+f(z,u)](\hat{u}_*-u)^+dz\ (\mbox{see (\ref{eq61}) and recall}\ c_{13}=c_{12}-\lambda>0)\\
			&=&\left\langle A_p(u)+A(u),(\hat{u}_*-u)^+\right\rangle\ (\mbox{since}\ u\in S_+(\lambda))\\
			&\Rightarrow&\left\langle A_p(\hat{u}_*)-A_p(u),(\hat{u}_*-u)^+\right\rangle+||D(\hat{u}_*-u)^+||^{2}_{2}\leq 0,\\
			&\Rightarrow&\hat{u}_*\leq u.
		\end{eqnarray*}
		
		Therefore we have proved that
$$\hat{u}_*\in[0,u]=\{y\in W^{1,p}_{0}(\Omega):0\leq y(z)\leq u(z)\ \mbox{for a.a.}\ z\in\Omega\},\ \hat{u}_*\neq 0.$$

So, (\ref{eq71}) becomes
\begin{eqnarray*}
	&&A_p(\hat{u}_*)+A(\hat{u}_*)=(\eta(z)-\delta)\hat{u}_*-c_{13}\hat{u}^{p-1}_{*},\\
	&\Rightarrow&\hat{u}_*\ \mbox{is a positive solution of problem (\ref{eq62})},\\
	&\Rightarrow&\hat{u}_*=u_*\in \mbox{int}\, C_+\ (\mbox{see Proposition \ref{prop24}}).
\end{eqnarray*}

Thus we have proved the claim.

Passing to the limit as $n\rightarrow\infty$ in (\ref{eq66}) and using (\ref{eq68}), we obtain
\begin{eqnarray*}
	 &&A_p(u^{*}_{\lambda})+A(u^{*}_{\lambda})=\lambda(u^{*}_{\lambda})^{p-1}+N_f(u^{*}_{\lambda}),\ u_*\leq u^{*}_{\lambda},\\
	&\Rightarrow&u^{*}_{\lambda}\in S_+(\lambda)\ \mbox{and}\ u^{*}_{\lambda}=\inf S_+(\lambda).
\end{eqnarray*}

For the biggest negative solution we use the set $S_-(\lambda)$ which is upward directed (that is, if $v_1,v_2\in S_-(\lambda)$, then we can find $v\in S_-(\lambda)$ such that $v_1\leq v,\ v_2\leq v$). Reasoning as above, we produce $v^{*}_{\lambda}\in S_-(\lambda)\subseteq-\mbox{int}\, C_+$ a biggest negative solution of \eqref{eqP}.
\end{proof}

Using these extremal constant sign solutions, we can produce a nodal  solution of problem \eqref{eqP}.

\begin{prop}\label{prop26}
	If hypothesis $H_4$ holds and $\lambda\in(\hat{\lambda}_1(p),\hat{\lambda}_1(p)+\epsilon)$ with $\epsilon>0$ as in Proposition \ref{prop23}, then problem \eqref{eqP} admits a nodal solution $y_0\in[v^{*}_{\lambda},u^{*}_{\lambda}]\cap C^{1}_{0}(\overline{\Omega})$.
\end{prop}
\begin{proof}
	Let $u^{*}_{\lambda}\in \mbox{int}\, C_+$ and $v^{*}_{\lambda}\in-\mbox{int}\, C_+$ be the extremal constant sign solutions of \eqref{eqP} produced in Proposition \ref{prop26}. We introduce the following truncation of the reaction in problem \eqref{eqP}
	\begin{eqnarray}\label{eq72}
		g_{\lambda}(z,x)=\left\{
		\begin{array}{ll}
			 \lambda|v^{*}_{\lambda}(z)|^{p-2}v^{*}_{\lambda}(z)+f(z,v^{*}_{\lambda}(z))&\mbox{if}\ x<v^{*}_{\lambda}(z)\\
			\lambda|x|^{p-2}x+f(z,x)&\mbox{if}\ v^{*}_{\lambda}(z)\leq x\leq u^{*}_{\lambda}(z)\\
			\lambda u^{*}_{\lambda}(z)^{p-1}+f(z,u^{*}_{\lambda}(z))&\mbox{if}\ u^{*}_{\lambda}(z)<x.
		\end{array}\right.
	\end{eqnarray}
	
	This is a Carath\'eodory function. We set $G_{\lambda}(z,x)=\int^{x}_{0}g_{\lambda}(z,s)ds$ and consider the $C^1$-functional $\hat{\varphi}_{\lambda}:W^{1,p}_{0}(\Omega)\rightarrow\RR$ defined by $$\hat{\varphi}_{\lambda}=\frac{1}{p}||Du||^{p}_{p}+\frac{1}{2}||Du||^{2}_{2}-\int_{\Omega}G_{\lambda}(z,u(z))dz\ \mbox{for all}\ u\in W^{1,p}_{0}(\Omega).$$
	
	Also, we introduce $g^{\pm}_{\lambda}(z,x)=g_{\lambda}(z,\pm x^{\pm})$, $G^{\pm}_{\lambda}(z,x)=\int^{x}_{0}g^{\pm}_{\lambda}(z,s)ds$ and the $C^1$-functional $\hat{\varphi}^{\pm}_{\lambda}:W^{1,p}_{0}(\Omega)\rightarrow\RR$ defined by
	 $$\hat{\varphi}^{\pm}_{\lambda}(u)=\frac{1}{p}||Du||^{p}_{p}+\frac{1}{2}||Du||^{2}_{2}-\int_{\Omega}G^{\pm}_{\lambda}(z,u(z))dz\ \mbox{for all}\ u\in W^{1,p}_{0}(\Omega).$$
	
	Reasoning as in the proof of Proposition \ref{prop25} and using (\ref{eq72}), we obtain
	$$K_{\hat{\varphi}_{\lambda}}\subseteq[v^{*}_{\lambda},u^{*}_{\lambda}],\ K_{\hat{\varphi}^{+}_{\lambda}}\subseteq[0,u^{*}_{\lambda}];\ K_{\hat{\varphi}^{-}_{\lambda}}\subseteq[v^{*}_{\lambda},0]$$
	
	The extremality of $u^{*}_{\lambda}\in \mbox{int}\, C_+$ and of $v^{*}_{\lambda}\in-\mbox{int}\, C_+$, implies that
	\begin{equation}\label{eq73}
		K_{\hat{\varphi}_{\lambda}}\subseteq[v^{*}_{\lambda},u^{*}_{\lambda}],\ K_{\hat{\varphi}^{+}_{\lambda}}=\{0,u^{*}_{\lambda}\},\ K_{\hat{\varphi}^{-}_{\lambda}}=\{v^{*}_{\lambda},0\}.
	\end{equation}
	
	\begin{claim}
		$u^{*}_{\lambda}$ and $v^{*}_{\lambda}$ are local minimizers of the functional $\hat{\varphi}_{\lambda}$.
	\end{claim}
		
		Clearly, $\hat{\varphi}^{+}_{\lambda}$ is coercive (see (\ref{eq72})). Also, it is sequentially weakly lower semicontinuous. So, by the Weierstrass theorem we can find $\hat{u}\in W^{1,p}_{0}(\Omega)$ such that
		\begin{equation}\label{eq74}
			\hat{\varphi}^{+}_{\lambda}(\hat{u})=\inf[\hat{\varphi}^{+}_{\lambda}(u):u\in W^{1,p}_{0}(\Omega)].
		\end{equation}
		
	As before hypothesis $H_4(iii)$ and the fact that $u^{*}_{\lambda}\in \mbox{int}\, C_+$ and $2<p$, imply that
	\begin{eqnarray*}
		&&\hat{\varphi}^{+}_{\lambda}(\pm\hat{u}_1(2))<0,\\
		&\Rightarrow&\hat{\varphi}^{+}_{\lambda}(\hat{u})<0=\hat{\varphi}^{+}_{\lambda}(0)\ (\mbox{see (\ref{eq73})}),\ \mbox{hence}\ \hat{u}\neq 0.
	\end{eqnarray*}
	
	From (\ref{eq74}) we have
	\begin{eqnarray*}
		&&\hat{u}\in K_{\hat{\varphi}^{+}_{\lambda}},\\
		&\Rightarrow&\hat{u}\in\{0,u^{*}_{\lambda}\},\ \hat{u}\neq 0,\\
		&\Rightarrow&\hat{u}=u^{*}_{\lambda}\in \mbox{int}\, C_+.
	\end{eqnarray*}
	
	Since $\left.\hat{\varphi}^{+}_{\lambda}\right|_{C_+}=\left.\hat{\varphi}_{\lambda}\right|_{C_+}$, it follows that $u^{*}_{\lambda}$ is a local $C^{1}_{0}(\overline{\Omega})$-minimizer of $\hat{\varphi}_{\lambda}$, hence it is a local $W^{1,p}_{0}(\Omega)$-minimizer of $\hat{\varphi}_{\lambda}$ (see Proposition \ref{prop7}).
	
	Similarly for $v^{*}_{\lambda}$, using this time the functional $\hat{\varphi}_{\lambda}$. This proves the claim.
	
	Without any loss of generality, we may assume that
	 $$\hat{\varphi}_{\lambda}(v^{*}_{\lambda})\leq\hat{\varphi}_{\lambda}(u^{*}_{\lambda}).$$
	
	The analysis is similar if the opposite inequality holds. We may assume that $K_{\hat{\varphi}_{\lambda}}$ is finite (otherwise we already have infinity many nodal solutions, see (\ref{eq73})). From the claim we know that $u^{*}_{\lambda}$ is a local minimizer of $\hat{\varphi}_{\lambda}$. So, we can find small $\rho\in(0,1)$ such that
	\begin{equation}\label{eq75} \hat{\varphi}_{\lambda}(v_*)\leq\hat{\varphi}_{\lambda}(u_*)<\inf[\hat{\varphi}_{\lambda}(u):||u-u^{*}_{\lambda}||=\rho]=m^{\lambda}_{\rho},\ ||v^{*}_{\lambda}-u^{*}_{\lambda}||>\rho
	\end{equation}
	(see Aizicovici, Papageorgiou and Staicu \cite{1} (proof of Proposition 22)).
	
	The functional $\hat{\varphi}_{\lambda}$ is coercive, hence it satisfies the $C$-condition (see \cite{27}). This fact and (\ref{eq75}) permit the use of Theorem \ref{th2} (the mountain pass theorem). So, we can find $y_0\in W^{1,p}_{0}(\Omega)$ such that
	\begin{equation}\label{eq76}
		y_0\in K_{\hat{\varphi}_{\lambda}}\subseteq[v^{*}_{\lambda},u^{*}_{\lambda}]\ \mbox{(see (\ref{eq73})) and}\ m^{\lambda}_{\rho}\leq\hat{\varphi}_{\lambda}(y_0).
	\end{equation}
	
	From (\ref{eq75}) and (\ref{eq76}) we have that $y_0\notin\{v^{*}_{\lambda},u^{*}_{\lambda}\}$ and $y_0$ is a solution of \eqref{eqP} (see (\ref{eq72})) with $y_0\in C^{1}_{0}(\overline{\Omega})$ (nonlinear regularity). We need to show that $y_0\neq 0$ in order to conclude that $y_0$ is nodal.
	
	Let $\rho=\max\{||u^{*}_{\lambda}||_{\infty},||v^{*}_{\lambda}||_{\infty}\}$ and let $\xi_{\rho}>0$ be as postulated by hypothesis $H_4(iv)$. Then
	\begin{eqnarray}\label{eq77}
		&&-\Delta_py_0(z)-\Delta y_0(z)+\xi_{\rho}(y_0(z))^{p-2}y_0(z)\nonumber\\
		&=&(\lambda+\xi_{\rho})|y_0(z)|^{p-2}y_0(z)+f(z,y_0(z))\nonumber\\
		&\leq&(\lambda+\xi_{\rho})u^{*}_{\lambda}(z)^{p-1}+f(z,u^{*}_{\lambda}(z))\ (\mbox{since}\ y_0\leq u^{*}_{\lambda},\ \mbox{see hypothesis} H_4(iv))\nonumber\\
		&=&-\Delta_pu^{*}_{\lambda}(z)-\Delta u^{*}_{\lambda}(z)+\xi_pu^{*}_{\lambda}(z)^{p-1}\ \mbox{a.e. in}\ \Omega.
	\end{eqnarray}
	
	As before (see the proof of Proposition \ref{prop24}), we consider the map $a:\RR^N\rightarrow \RR^N$ defined by
	\begin{eqnarray*}
		&&a(y)=|y|^{p-2}y+y\ \mbox{for all}\ y\in\RR^N,\\
		&\Rightarrow&\nabla a(y)=|y|^{p-2}\left(I+(p-2)\frac{y\otimes y}{|y|^2}\right)+I,\\
		&\Rightarrow&(\nabla a(y)\xi,\xi)_{\RR^N}\geq|\xi|^2\ \mbox{for all}\ y,\xi\in\RR^N.
	\end{eqnarray*}
	
	So, we can apply the tangency principle of Pucci and Serrin \cite[p. 35]{29}, and obtain
	$$y_0(z)<u^{*}_{\lambda}(z)\ \mbox{for all}\ z\in\Omega.$$
	
	Then from (\ref{eq77}) and Proposition 2.6 of Arcoya and Ruiz \cite{3}, we have
	$$u^{*}_{\lambda}-y_0\in \mbox{int}\, C_+.$$
	
	In a similar fashion, we show that
	$$y_0-v^{*}_{\lambda}\in \mbox{int}\, C_+.$$
	
	So, we have proved that
	\begin{equation}\label{eq78}
		y_0\in \mbox{int}_{C^{1}_{0}(\overline{\Omega})}[v^{*}_{\lambda},u^{*}_{\lambda}].
	\end{equation}
	
	We consider the deformation
	$$h(t,u)=h_t(u)=(1-t)\hat{\varphi}_{\lambda}(u)+t\varphi_{\lambda}(u)\ \mbox{for all}\ (t,u)\in[0,1]\times W^{1,p}_{0}(\Omega)$$
	
	Suppose we can find $\{t_n\}_{n\geq 1}\subseteq[0,1]$ and $\{u_n\}_{n\geq 1}\subseteq W^{1,p}_{0}(\Omega)$ such that
	\begin{eqnarray}\label{eq79}
		&&t_n\rightarrow t\ \mbox{in}\ [0,1],\ u_n\rightarrow y_0\ \mbox{in}\ W^{1,p}_{0}(\Omega)\ \mbox{as}\ n\rightarrow\infty\ \mbox{and}\ (h_{t_n})'_u(t_n,u_n)=0\\
		&&\hspace{9.5cm}\mbox{for all}\ n\geq 1.\nonumber
	\end{eqnarray}
	
	We have
	\begin{eqnarray*}
		&&A_p(u_n)+A(u_n)=(1-t_n)N_{g_{\lambda}}(u_n)+t_n\lambda|u_n|^{p-2}u_n+t_nN_f(u_n)\ n\geq 1\\
		&\Rightarrow&-\Delta_pu_n(z)-\Delta u_n(z)=(1-t_n)g_{\lambda}(z,u_n(z))+t_n\lambda|u_n(z)|^{p-2}u_n(z)+t_nf(z,u_n(z))\\
		&&\hspace{9.5cm}\mbox{for a.a.}\ z\in \Omega.
	\end{eqnarray*}
	
	From Ladyzhenskaya and Uraltseva \cite[p. 286]{15}, we know that there exists $M_7>0$ such that
	$$||u_n||_{\infty}\leq M_7\ \mbox{for all}\ n\geq 1.$$
	
	Hence by virtue of Theorem 1 of Lieberman \cite{16}, there exists $\alpha\in(0,1)$ and $M_8>0$ such that
	$$u_n\in C^{1,\alpha}_{0}(\overline{\Omega})\ \mbox{and}\ ||u_n||_{C^{1,\alpha}_{0}(\overline{\Omega})}\leq M_8\ \mbox{for all}\ n\geq 1.$$
	
	Exploiting the compact embedding of $C^{1,\alpha}_{0}(\overline{\Omega})$ into $C^{1}_{0}(\overline{\Omega})$ and using (\ref{eq79}), we have
	\begin{eqnarray*}
		&&u_n\rightarrow y_0\ \mbox{in}\ C^{1}_{0}(\overline{\Omega})\ \mbox{as}\ n\rightarrow\infty,\\
		&\Rightarrow&u_n\in[v^{*}_{\lambda},u^{*}_{\lambda}]\ \mbox{for all}\ n\geq n_0\ (\mbox{see (\ref{eq78})}).
	\end{eqnarray*}
	
	But from (\ref{eq72}) we see that $\{u_n\}_{n\geq 1}\subseteq K_{\varphi_{\lambda}}$, a contradiction to our hypotheses that $K_{\varphi_{\lambda}}$ is finite. So, (\ref{eq78}) cannot happen and hence the homotopy invariance of singular homology implies that
	\begin{equation}\label{eq80}
		C_k(\varphi_{\lambda},y_0)=C_k(\hat{\varphi}_{\lambda},y_0)\ \mbox{for all}\ k\geq 0.
	\end{equation}
	
	Recall that $y_0$ is a critical point of mountain pass type the functional $\hat{\varphi}_{\lambda}$. Therefore
	\begin{eqnarray}\label{eq81}
		&&C_1(\hat{\varphi}_{\lambda},y_0)\neq 0,\nonumber\\
		&\Rightarrow&C_1(\varphi_{\lambda},y_0)\neq 0\ (\mbox{see (\ref{eq80})}).
	\end{eqnarray}
	
	From Proposition \ref{prop13}, we know that
	\begin{eqnarray*}
		&&C_k(\varphi_{\lambda},0)=\delta_{k,d_m}\ZZ\ \mbox{for all}\ k\geq 0,\ \mbox{with}\ d_m\geq 2,\\
		&\Rightarrow&y_0\neq 0\ (\mbox{see (\ref{eq81})}),\\
		&\Rightarrow&y_0\in C^{1}_{0}(\overline{\Omega})\ \mbox{is nodal}.
	\end{eqnarray*}
\end{proof}

So, we can state the following multiplicity theorem for problem \eqref{eqP}.
\begin{theorem}\label{th27}
	If hypothesis $H_4$ holds, then there exists $\epsilon>0$ such that for all\\ $\lambda\in(\hat{\lambda}_1(p),\hat{\lambda}_1(p)+\epsilon)$ problem \eqref{eqP} has at least three nontrivial solutions
	$$u_0\in {\rm int}\, C_+,\ v_0\in-{\rm int}\, C_+\ \mbox{and}\ y_0\in int_{C^{1}_{0}(\overline{\Omega})}[v_0,u_0]\ \mbox{is nodal}.$$
\end{theorem}
\begin{remark}
	We stress that the above theorem provides sign information for all solutions and localizes them. None of the other papers mentioned in the introduction, contains such a multiplicity result for equations near resonance from above.
\end{remark}
	
	In fact we can improve Theorem \ref{th27} and generate a second nodal solution provided we strengthen the regularity of $f(z,\cdot)$.
		The new hypotheses on $f(z,x)$ are the following:
	
\smallskip
	$H_5:$ $f:\Omega\times\RR\rightarrow\RR$ is a measurable function such that for a.a. $z\in\Omega$, $f(z,0)=0,\ f(z,\cdot)\in C^1(\RR)$ and
	\begin{itemize}
		\item[(i)] $|f'_x(z,x)|\leq a(z)(1+|x|^{p-2})$ for a.a. $z\in\Omega$, all $x\in\RR$ with $a\in L^{\infty}(\Omega)_+$;
		\item[(ii)] there exists $\vartheta\in L^{\infty}(\Omega)$ such that $\vartheta(z)\leq 0$ for a.a. $z\in\Omega,\ \vartheta\neq 0$ and
		$$\limsup\limits_{x\rightarrow\pm\infty}\ \frac{pF(z,x)}{|x|^p}\leq\vartheta(z)\ \mbox{uniformly for a.a.}\ z\in\Omega;\ \mbox{and}$$
		\item[(iii)] there exists integer $m\geq 2$ such that
		\begin{eqnarray*}
			&&f'_x(z,0)\in[\hat{\lambda}_m(2),\hat{\lambda}_{m+1}(2)]\ \mbox{for a.a.}\ z\in\Omega,\ f'_x(\cdot,0)\not\equiv\hat{\lambda}_m(2),f'_x(\cdot,0)\not\equiv\hat{\lambda}_{m+1}(2)\\
			&&f'_x(z,0)=\lim\limits_{x\rightarrow 0}\ \frac{f(z,x)}{x}\ \mbox{uniformly for a.a.}\ z\in\Omega.
		\end{eqnarray*}
	\end{itemize}

\begin{remark}
	The differentiability of $f(z,\cdot)$ and hypothesis $H_5(i)$ imply that for every $\rho>0$, there exists $\xi_{\rho}>0$ for a.a. $z\in\Omega$, $x\rightarrow f(z,x)+\xi_{\rho}|x|^{p-2}x$ is nondecreasing on $[-\rho,\rho]$.
\end{remark}

We can now state the following multiplicity theorem.
\begin{theorem}\label{th28}
	If hypothesis $H_5$ holds, then there exists $\epsilon>0$ such that for all\\ $\lambda\in(\hat{\lambda}_1(p),\hat{\lambda}_1(p)+\epsilon)$ problem \eqref{eqP} admits at least four nontrivial solutions
	\begin{eqnarray*}
		&&u_0\in {\rm int}\, C_+,\ v_0\in-{\rm int}\, C_+\\
		\mbox{and}&&y_0,\hat{y}\in {\rm int}_{C^{1}_{0}(\overline{\Omega})}[v_0,u_0]\ \mbox{is nodal}.
	\end{eqnarray*}
\end{theorem}
\begin{proof}
	From Theorem \ref{th27} we already know that there exists $\epsilon>0$ such that for all $\lambda\in(\hat{\lambda}_1(p),$\\ $\hat{\lambda}_1(p)+\epsilon)$ has at least three nontrivial solutions
	$$u_0\in \mbox{int}\, C_+,\ v_0\in-\mbox{int}\, C_+\ \mbox{and}\ y_0\in \mbox{int}_{C^{1}_{0}(\overline{\Omega})}[v_0,u_0]\ \mbox{is nodal}.$$
	
	By virtue of Proposition \ref{prop25}, we may assume that $u_0$ and $v_0$ are extremal (that is, $u_0=u^{*}_{\lambda}\in \mbox{int}\, C_+$ and $v_0=v^{*}_{\lambda}\in-\mbox{int}\, C_+$). From the proof of Proposition \ref{prop26} (see the claim), we know that $u_0$ and $v_0$ are local minimizers of the functional $\hat{\varphi}_{\lambda}$. Therefore
	\begin{equation}\label{eq82}
		C_k(\hat{\varphi}_{\lambda},u_0)=C_k(\hat{\varphi_{\lambda}},v_0)=\delta_{k,0}\ZZ\ \mbox{for all}\ k\geq 0.
	\end{equation}
	
	Since $\left.\hat{\varphi}_{\lambda}\right|_{[v_0,u_0]}=\left.\varphi_{\lambda}\right|_{[v_0,u_0]}$ (see (\ref{eq72})) and since $v_0\in -\mbox{int}\, C_+$, $u_0\in \mbox{int}\, C_+$ from Proposition \ref{prop13}, we have
	\begin{equation}\label{eq83}
		C_k(\hat{\varphi}_{\lambda},0)=\delta_{k,d_m}\ZZ\ \mbox{for all}\ k\geq 0,\ \mbox{with}\ d_m\geq 2.
	\end{equation}
	
	From the proof of Proposition \ref{prop26}, we have
	\begin{eqnarray}\label{eq84}
		&&C_k(\hat{\varphi}_{\lambda},y_0)\neq 0,\nonumber\\
		&\Rightarrow&C_k(\hat{\varphi}_{\lambda},y_0)=\delta_{k,1}\ZZ\ \mbox{for all}\ k\geq 0\\
		&&(\mbox{see Papageorgiou and Smyrlis \cite{26} and Papageorgiou and R\u{a}dulescu \cite{23}}).\nonumber
	\end{eqnarray}
	
	Finally, since $\hat{\varphi}_{\lambda}$ is coercive (see (\ref{eq72})), we have
	\begin{equation}\label{eq85}
		C_k(\hat{\varphi}_{\lambda},\infty)=\delta_{k,0}\ZZ\ \mbox{for all}\ k\geq 0.
	\end{equation}
	
	Suppose $K_{\hat{\varphi}_{\lambda}}=\{u_0,v_0,0,y_0\}$. From (\ref{eq82}), (\ref{eq83}), (\ref{eq84}), (\ref{eq85}) and the Morse relation with $t=-1$ (see (\ref{eq4})), we have
	\begin{eqnarray*}
		&&2(-1)^0+(-1)^{d_m}+(-1)^1=(-1)^0,\\
		&\Rightarrow&(-1)^{d_m}=0,\ \mbox{a contradiction.}
	\end{eqnarray*}
	
	So, we have $\hat{u}\in K_{\hat{\varphi}_{\lambda}}\subseteq[v_0,u_0]$ (see (\ref{eq73})), $\hat{y}\notin\{u_0,v_0,0\}$, thus $\hat{y}$ is nodal. Moreover, from the nonlinear regularity theory and reasoning as before (see the proof of Proposition \ref{prop26}), we have
	$$\hat{y}\in \mbox{int}_{C^{1}_{0}(\overline{\Omega})}[v_0,u_0].$$
\end{proof}

\medskip
{\bf Acknowledgments.}
This research was supported by
 the Slovenian Research Agency grants P1-0292,  J1-7025, and J1-6721, and
 the Romanian Research Council  grant CNCS-UEFISCDI-PCCA-23/2014.

\end{document}